\documentclass[a4paper,11pt]{amsart}
\usepackage{amssymb,amsmath,amsthm}
\usepackage{graphicx,color}
\usepackage{stackrel}
\usepackage{hyperref}
\newtheorem{Lemma}{Lemma}
\newtheorem{Theorem}{Theorem}
\newtheorem{Corollary}{Corollary}
\newtheorem{Proposition}{Proposition}

\newtheorem{Definition}{Definition}
\newtheorem{Example}{Example}

\theoremstyle{definition}
\newtheorem{Remark}{Remark}

\numberwithin{equation}{section}
\usepackage[normalem]{ulem}

\definecolor{DarkGreen}{rgb}{0,0.5,0.1}

\newcommand\soutD{\bgroup\markoverwith
{\textcolor{DarkGreen}{\rule[.5ex]{2pt}{1pt}}}\ULon}
\newcommand\soutP{\bgroup\markoverwith
{\textcolor{blue}{\rule[.5ex]{2pt}{1pt}}}\ULon}
\newcommand{\Hm}[1]{\leavevmode{\marginpar{\tiny%
$\hbox to 0mm{\hspace*{-0.5mm}$\leftarrow$\hss}%
\vcenter{\vrule depth 0.1mm height 0.1mm width \the\marginparwidth}%
\hbox to
0mm{\hss$\rightarrow$\hspace*{-0.5mm}}$\\\relax\raggedright #1}}}

\begin{document}
%
%-------%
% TITLE %
% TITLE %
%-------%
%------------------------------------------%
%------------------------------------------%
\title{Uncertainty principle via variational calculus on modulation spaces}

\author[Dias, Luef and Prata]{Nuno C. Dias, Franz Luef and Jo\~{a}o N. Prata}

\address{N.C. Dias - Grupo de F\'{\i}sica Matem\'{a}tica, Faculdade de Ci\^{e}ncias da Universidade de Lisboa, Campo Grande, Edif\'{\i}cio C6
1749-016 Lisboa, Portugal, and Escola Superior N\'autica Infante D. Henrique, Av. Engenheiro Bonneville Franco, 2770-058 Pa\c{c}o de Arcos, Portugal} \email{ncdias@meo.pt}

\address{F. Luef - Department of Mathematical Sciences, NTNU Trondheim, 7041 Trondheim, Norway}
\email{franz.luef@ntnu.no}

\address{J.N. Prata - Grupo de F\'{\i}sica Matem\'{a}tica, Faculdade de Ci\^{e}ncias da Universidade de Lisboa, Campo Grande, Edif\'{\i}cio C6
1749-016 Lisboa, Portugal, and Escola Superior N\'autica Infante D. Henrique, Av. Engenheiro Bonneville Franco, 2770-058 Pa\c{c}o de Arcos, Portugal}
\email{joao.prata@mail.telepac.pt}

\date{\today}
%\date{May 26, 2015}

\begin{abstract}
We approach uncertainty principles of Cowling-Price-Heis-\\enberg-type as a variational principle on modulation
spaces. In our discussion we are naturally led to compact localization operators with symbols in
modulation spaces. The optimal constant in these uncertainty principles is the smallest
eigenvalue of the inverse of a compact localization operator. The Euler-Lagrange equations for
the associated functional provide equations for the eigenfunctions of the smallest eigenvalue
of these compact localization operators. As a by-product of our proofs we derive a generalization to mixed-norm spaces
of an inequality for Wigner and Ambiguity functions due do Lieb.
\end{abstract}

\maketitle

\section{Introduction}

A cornerstone of Fourier analysis is the uncertainty principle about the time and frequency localization of a function. Loosely speaking it states that a function $f(x)$ and its Fourier transform $\widehat{f} (\omega)$ cannot both be sharply localized. One of the first quantitative versions of the uncertainty principle to be put forward was the Heisenberg-Pauli-Weyl inequality:
\begin{equation}
\left(\int_{\mathbb{R}} (x-x_0)^2 |f(x)|^2 dx \right)^{1/2} \left(\int_{\mathbb{R}} (\omega-\omega_0)^2 |\widehat{f}(\omega)|^2 d\omega \right)^{1/2} \geq \frac{\|f\|_{L^2 (\mathbb{R})}^2}{4 \pi},
\label{eqIntro1}
\end{equation}
for any $x_0,\omega_0 \in \mathbb{R}$. In what follows, we will set $x_0= \omega_0=0$, as this can be easily achieved by a phase-space translation, which does not alter the results.

This inequality can be shown to be equivalent to \cite{Folland}:
\begin{equation}
\|xf\|_{L^2(\mathbb{R})}^2 + \|\omega \widehat{f} \|_{L^2(\mathbb{R})}^2 \geq \frac{\|f\|_{L^2(\mathbb{R})}^2}{2 \pi}.
\label{eqIntro2}
\end{equation}
The constant is sharp and we obtain an equality if and only if $f$ is a generalized Gaussian \cite{Grochenig}.

Cowling and Price \cite{Cowling} have generalized (\ref{eqIntro2}) by considering other norms and other weights. Here we wish to consider the following result \cite{Cowling,Folland}.
\begin{Theorem}[Cowling-Price]\label{TheoremCowlingPrice}
Suppose $p, q \in \left[1, \infty\right]$ and $a,b>0$. There is a constant $K>0$ such that:
\begin{equation}
\|~|x|^a f\|_{L^p(\mathbb{R})} + \|~|\omega|^b \widehat{f} \|_{L^q(\mathbb{R})} \geq K \|f\|_{L^2(\mathbb{R})} ,
\label{eqIntro3}
\end{equation}
for all tempered distributions $f$ and $\widehat{f}$, which are  represented by locally integrable functions, if and only if:
\begin{equation}
a>\frac{1}{2} - \frac{1}{p} \hspace{0.5 cm} \text{ and } \hspace{0.5 cm}  b>\frac{1}{2} - \frac{1}{q} ~ .
\label{eqIntro4}
\end{equation}
\end{Theorem}

In the present work we intend to extend the Cowling-Price uncertainty principle in the following sense. We shall prove, in arbitrary dimension $d\geq 1$, that inequalities of the form
\begin{equation}
\|~|x|^a f\|_{L^p(\mathbb{R}^d)} + \|~|\omega|^b \widehat{f} \|_{L^q(\mathbb{R}^d)} \geq K \|f\|_{M_{\alpha , \beta}^{r,s}(\mathbb{R}^d)} ~,
\label{eqIntro5}
\end{equation}
hold for all $f \in M_{\alpha , \beta}^{r,s}(\mathbb{R}^d)$ and where the left-hand side may be infinite. Here $M_{\alpha , \beta}^{r,s}(\mathbb{R}^d)$ is a certain modulation space to be specified below. Moreover, we shall prove that there exist functions $f_0$ for which we obtain an equality and which are the solutions of certain nonlinear functional equations. This latter result is missing in Cowling and Price's original work.

In addition, we shall also prove that, for certain weight functions $\psi, \phi$, we have an inequality of the form:
\begin{equation}
\|\psi f\|_{L^2(\mathbb{R}^d)}^2 + \|\phi \widehat{f} \|_{L^2(\mathbb{R}^d)}^2 \geq K \|f\|_{M_{m_0}^2(\mathbb{R}^d)} ~,
\label{eqIntro6}
\end{equation}
where $M_{m_0}^2(\mathbb{R}^d)$ is another modulation space.

Inequalities (\ref{eqIntro5},\ref{eqIntro6}) can be regarded as generalizations of the Heisenberg-Pauli-Weyl inequality (\ref{eqIntro2}). Moreover, we shall show that we obtain an equality in (\ref{eqIntro6}) for functions $f_0$ which are eigenfunctions of a certain positive compact localization operator and that the optimal constant is related with its eigenvalues.

One ingredient of our treatment is the work of Galperin and Gr\"ochenig \cite{Galperin,GalperinGrochenig,Grochenig1}, where Heisenberg-type inequalities are interpreted as embeddings of weighted Lebesgue spaces $L^p_a(\mathbb{R}^d)$ and weighted Fourier images $\mathcal{F}L^q_b(\mathbb{R}^d)$ of weighted $L^q_b(\mathbb{R}^d)$ spaces into modulation spaces $M^{r,s}_{\alpha,\beta}(\mathbb{R}^d)$.

We denote by $L^p_a(\mathbb{R}^d)$ the space of distributions $f$ in the Lebesgue space $L^p(\mathbb{R}^d)$ with finite $\|.\|_{L^p_a (\mathbb{R}^d)}$, where $\|f\|_{L^p_a (\mathbb{R}^d)}=\big(\int_{\mathbb{R}^d}|f(x)|^p(1+|x|)^{ap}dx\big)^{1/p}$.

The space $\mathcal{F}L^q_b(\mathbb{R}^d)$ consists of all tempered distributions $f\in\mathcal{S}^\prime(\mathbb{R}^d)$ such that $\widehat{f}$ is in $L^q_b(\mathbb{R}^d)$ with finite norm $\|f\|_{\mathcal{F}L^q_b (\mathbb{R}^d)}=\|\widehat{f}\|_{L^q_b (\mathbb{R}^d)}$.

Modulation spaces were introduced by Feichtinger in $1983$ \cite{Feichtinger1,Feichtinger2,Feichtinger3}. Here we restrict the discussion to the following classes of modulation spaces:
$M^{r,s}_{\alpha,\beta}(\mathbb{R}^d)$ and $M^2_m(\mathbb{R}^d)$, to be defined below. Loosely speaking, modulation spaces
allow one to encode the decay and integrability properties of a function on the phase-space.

There are various ways to measure the phase-space content of a function,
e.g. Wigner transform, Rihazcek transform \cite{Cohen}. We use the matrix coefficients of the Schr\"odinger representation $\pi$ of of the Heisenberg group $\mathbb{H} (d)$, $\pi(x,\omega)g(t)=e^{2\pi it \cdot \omega}g(t-x)$:
\begin{equation}
  V_gf(x,\omega)=\left( f,\pi(x,\omega)g \right)_{L^2}=\int_{\mathbb{R}^d} f(t)\overline{g(t-x)}e^{-2\pi it \cdot \omega}dt,
\label{eqSTFT}
\end{equation}
for $f,g$ in $L^2(\mathbb{R}^d)$. In signal analysis $V_gf$ is known as the short-time-Fourier transform. If $g \in \mathcal{S} (\mathbb{R}^d)$, then the short-time Fourier transform extends to tempered distributions $f \in \mathcal{S}^{\prime} (\mathbb{R}^d)$.

For $\alpha,\beta \in \mathbb{R} $ and $r,s\in[1,2]$ we define the modulation space $M^{r,s}_{\alpha,\beta}(\mathbb{R}^d)$ as the space of all functions
such that
\begin{equation}
  \|f\|_{M^{r,s}_{\alpha,\beta} (\mathbb{R}^d)}=\Big(\int_{\mathbb{R}^d} \Big(\int_{\mathbb{R}^d} |V_gf(x,\omega)|^r(1+|x|)^{\alpha r}dx\Big)^{s/r}(1+|\omega|)^{\beta s}d\omega\Big)^{1/s}
\label{eqModulationSpace1}
\end{equation}
is finite with respect to a fixed non-zero $g$ in $\mathcal{S}(\mathbb{R}^d)$. Note that different choices of the non-zero window $g\in\mathcal{S}(\mathbb{R}^d)$ lead to equivalent norms for $M^{r,s}_{\alpha,\beta}(\mathbb{R}^d)$. If we want to emphasize the choice of window $g$, we may sometimes write $ \| \cdot \|_{g, M^{r,s}_{\alpha,\beta} (\mathbb{R}^d)}$.

We shall also consider the modulation spaces with norm
\begin{equation}
  \|f\|_{M^2_m (\mathbb{R}^d)}=\Big(\int_{\mathbb{R}^d} \int_{\mathbb{R}^d} |m (x,\omega) V_gf(x,\omega)|^2 dx d \omega\Big)^{1/2},
\label{eqModulationSpace1.1}
\end{equation}
where $m (x, \omega)$ is an appropriate time-frequency weight (see e.g.\cite{Aceska}).

We need the following particular case of the results in \cite{Galperin,GalperinGrochenig}:
\begin{Theorem}[Galperin-Gr\"ochenig]\label{TheoremGalperinGrocheing}	
  Let $\alpha,\beta\ge 0$, $0<r,s\le 2$ and $1 \le p,q \le \infty$. Suppose that $r\le p$ and $s\le q$.
If
	\begin{equation}
\begin{array}{c}
	  \left(\frac{a-\alpha}{d} +\frac{1}{p}-\frac{1}{r} \right) \left(\frac{b-\beta}{d} +\frac{1}{q}-\frac{1}{s} \right)> \\
\\
\mbox{max} \left\{\left(\frac{1}{r} - \frac{1}{q^{\prime}} + \frac{\alpha}{d} \right), \left(\frac{1}{r} - \frac{1}{2} + \frac{\alpha}{d} \right) \right\} \times  \mbox{max} \left\{\left(\frac{1}{s} - \frac{1}{p^{\prime}} + \frac{\beta}{d} \right), \left(\frac{1}{s} - \frac{1}{2} + \frac{\beta}{d} \right) \right\}
	\end{array}
\label{eqGalperinGrochenig1}
\end{equation}
	with all factors positive, then $L^p_a(\mathbb{R}^d)\cap\mathcal{F}L^q_b(\mathbb{R}^d)$ is compactly embedded in $M^{r,s}_{\alpha,\beta}(\mathbb{R}^d)$. Thus, there exists $C>0$ such that:
	\begin{equation}
	  \|f\|_{L^p_a (\mathbb{R}^d)}+\|\widehat{f}\|_{L^q_b (\mathbb{R}^d)}\ge C\|f\|_{M^{r,s}_{\alpha,\beta}(\mathbb{R}^d)}.
	\label{eqGalperinGrochenig2}
\end{equation}
\end{Theorem}
As usual $p^{\prime}, q^{\prime}, \cdots$ denote the H\"older duals of $p,q, \cdots \in \left[1, \infty \right]$: $\frac{1}{p}+ \frac{1}{p^{\prime}}=1$, {\it etc}. We also remark that the condition $r^{\prime} \leq p, p^{\prime} \leq r $ is equivalent to $\left|\frac{1}{p}- \frac{1}{2}\right| \leq \left|\frac{1}{r}- \frac{1}{2} \right|$.

We will resort to this theorem to prove (Theorem \ref{TheoremMinimum}) that inequality (\ref{eqIntro5}) holds if $1<r,s \leq 2$; $\alpha, \beta \geq 0$; $0<a,b \leq r^{\prime}$, and $r \leq p \leq r^{\prime}$; $s \leq q \leq r^{\prime}$ and (\ref{eqGalperinGrochenig1}) is valid.

That this extends the Cowling-Price uncertainty principle is now obvious. If we set $r=s=2$, $\alpha=\beta=0$, $d=1$ in (\ref{eqGalperinGrochenig1}) , we obtain (\ref{eqIntro4}). Likewise (\ref{eqIntro3}) follows from (\ref{eqIntro5}) and the fact that $M_{0,0}^{2,2} (\mathbb{R})=L^2 (\mathbb{R})$ \cite{Grochenig}.

 We shall prove the second inequality (eq.(\ref{eqIntro6})) in Theorem \ref{TheoremMinimizersHilbert}. For this, we shall require another compact embedding theorem  due to Boggiatto and Toft \cite{Boggiatto} (see also \cite{Pfeuffer}) for modulation spaces of the form $M_m^{p,q} (\mathbb{R}^d)$ (to be defined below in eq.(\ref{eqModulationspace3})):
\begin{Theorem}[Boggiatto-Toft]\label{TheoremBoggiatto}
Assume that $m_1,m_2 \in \mathcal{P} (\mathbb{R}^{2d})$, and that $p,q \in \left[1, \infty \right]$. Then the embedding
\begin{equation}
i: M_{m_1}^{p,q} (\mathbb{R}^d) \to  M_{m_2}^{p,q} (\mathbb{R}^d)
\label{eqBoggiatto1}
\end{equation}
is compact if and only if $m_2 / m_1 \in L_0^{\infty} (\mathbb{R}^{2d})$.
\end{Theorem}

Here $ \mathcal{P} (\mathbb{R}^{2d})$ is a class of weights which are polynomially moderate (see Definition \ref{DefinitionWeight0}) and $L_0^{\infty} (\mathbb{R}^{2d})$ is the set of all $f \in L^{\infty} (\mathbb{R}^{2d})$ such that
\begin{equation}
\lim_{R \to \infty} \left( \mbox{ess~sup}_{|z| \ge R} |f(z)| \right)=0 .
\label{eqweights1}
\end{equation}
Let us remark that the compactness of such embeddings between modulation spaces also follows from Theorem 9.4 in \cite{Feichtinger5}.

The strategy of the proofs consists of approaching this problem from the point of view of variational calculus. We will minimize the functional
\begin{equation}
f \mapsto \|~|x|^a f \|_{L^p(\mathbb{R}^d)} + \|~|\omega|^b \widehat{f} \|_{L^q(\mathbb{R}^d)}
\label{eqIntro7}
\end{equation}
on the constraint set:
\begin{equation}
\Omega= \left\{f \in L_a^p (\mathbb{R}^d) \cap \mathcal{F} L_b^q (\mathbb{R}^d): ~ \|f\|_{M_{\alpha, \beta}^{r,s} (\mathbb{R}^d)}=1 \right\}.
\label{eqIntro8}
\end{equation}
In a similar fashion, we will seek to minimize
\begin{equation}
f \mapsto  \| \psi f \|_{L^2(\mathbb{R}^d)}^2 + \| \phi \widehat{f} \|_{L^2(\mathbb{R}^d)}^2,
\label{eqIntro9}
\end{equation}
for $f$ in an appropriate modulation space $M_m^2 (\mathbb{R}^d)$ and such that $\|f\|_{M_{m_0}^2 (\mathbb{R}^d)}=1$, where $M_{m_0}^2 (\mathbb{R}^d)$ is another modulation space, such that $M_{m}^2 (\mathbb{R}^d) \subset M_{m_0}^2 (\mathbb{R}^d)$ (see Corollary \ref{CorollaryEmbedding}).

These variational problems will naturally lead to Euler-Lagrange equations (Theorem \ref{TheoremBanachCase} and Theorem \ref{TheoremEulerLagrangeEquation1}, Theorem \ref{TheoremEulerLagrangeInverse}).

To prove the existence of minimizers, we start by showing that the functionals (\ref{eqIntro7},\ref{eqIntro9}) are weakly lower semicontinuous (Lemma \ref{LemmaWeakLowerSemicontinuity} and Lemma \ref{Lemmaweaklowersemicontinuous}). Then we prove that certain subsets of the constraint sets \footnote{Each of these subsets is the intersection of a closed ball and the constraint set.} are weakly sequentially compact (Proposition \ref{PropositionWeakCompactness} and Proposition \ref{PropositionCompactnessHilbert1}). This means that the functionals attain a minimum on these subsets. Finally, we show that they are in fact minima on the entire constraint sets (Theorems \ref{TheoremMinimum} and \ref{TheoremMinimizersHilbert}).

An essential tool in this last step is a generalization of an inequality for Wigner and radar ambiguity functions due to E. Lieb. In \cite{Lieb} Lieb proved that
\begin{equation}
\|A(f,g) \|_{L^r (\mathbb{R}^2)} \geq C \|f\|_{L^p (\mathbb{R})} ~ \|g\|_{L^{p^{\prime}} (\mathbb{R})},
\label{eqIntro10}
\end{equation}
where $1 \leq r <2$; $r \leq p, p^{\prime} \leq r^{\prime}$, $C>0$ is a constant and $A(f,g) (x, \omega)$ is the radar ambiguity function:
\begin{equation}
A(f,g) (x, \omega)=e^{-i \pi \omega \cdot x} V_g f (-x, \omega).
\label{eqIntro11}
\end{equation}
In Theorem \ref{Bounds} and Corollary \ref{CorollaryBoundsA} we shall prove an extension of this inequality to arbitrary dimension $d \geq 1$ and mixed-norm spaces:
\begin{equation}
\|A(f,g) \|_{L_{\omega,x}^{r,s} (\mathbb{R}^{2d})} \geq C \|f\|_{L^u (\mathbb{R}^d)} ~ \|g\|_{L^v (\mathbb{R}^d)}
\label{eqIntro12}
\end{equation}
for appropriate values of $u,v,r,s$ and where
\begin{equation}
||F||_{L_{\omega,x}^{r,s} (\mathbb{R}^{2d})}=\left(\int_{\mathbb{R}^d} \left(\int_{\mathbb{R}^d} |F(x, \omega)|^r d \omega \right)^{\frac{s}{r}} d x\right)^{\frac{1}{s}}~.
\label{eqIntro13}
\end{equation}
We also determine the sharp constants and the functions which saturate (\ref{eqIntro12}).

\section*{Notation}

We denote by $\mathcal{S}(\mathbb{R}^d)$ the space of Schwartz test functions and by $\mathcal{S}^{\prime}(\mathbb{R}^d)$ its dual - the space of tempered distributions. The duality bracket is $\langle f,g \rangle$ for $g$ in some functional space $\mathcal{B}$ and $f$ in its dual $\mathcal{B}^{\prime}$. The inner product in a Hilbert space $\mathcal{H}$ is $\left( \cdot, \cdot \right)_{\mathcal{H}}$, which is linear in the first argument. The Fourier transform of $f \in L^2 (\mathbb{R}^d)$ is given by
\begin{equation}
\widehat{f} (\omega) = \int_{\mathbb{R}^d} f(x) e^{- 2 \pi i x \cdot \omega} dx
\label{eqNotation1}
\end{equation}
where
\begin{equation}
x \cdot \omega = x_1 \omega_1+ \cdots x_d \omega_d,
\label{eqNotation2}
\end{equation}
for all $x=(x_1, \cdots, x_d),~ \omega=\left(\omega_1, \cdots, \omega_d \right) \in \mathbb{R}^d$. We shall also use the notation:
\begin{equation}
\langle x \rangle = 1+ |x|.
\label{eqNotation2.1}
\end{equation}
Notice that the notation $\langle x \rangle $ is more commonly used in the literature for the Japanese bracket $\langle x \rangle =\left(1+|x|^2 \right)^{1/2}$, which is equivalent to our weight (\ref{eqNotation2.1}). We shall nevertheless use the weight (\ref{eqNotation2.1}) as it will make our derivations simpler.

If there is a constant $C>0$ such that $ A(f) \le C B(f)$ for all $f$ in some set, then we write $A(f)\lesssim B(f)$. If  $A(f)  \lesssim B(f)$ and $B(f) \lesssim A(f)$, then we write  $A(f) \asymp B(f)$. If a sequence $(u_n)_n$ in some normed space $X$ converges strongly to some $u \in X$, then we write $u_n \to u$, and if it converges weakly, then we write $u_n \rightharpoonup u$.

We denote the compact embedding of a functional space $\mathcal{A}$ into $\mathcal{B}$ by $\mathcal{A} \subset \subset \mathcal{B}$.

\section{Modulation spaces}

\begin{Definition}\label{DefinitionWeight0}
A weight in $\mathbb{R}^d$ is a positive and continuous function $m \in L_{loc}^{\infty} (\mathbb{R}^d)$. Given two weights $m$ and $v$, $m$ is said to be $v$-moderate, if
\begin{equation}
m(x+y) \le C m(x) v(y), \hspace{1 cm} \forall x,y \in \mathbb{R}^d,
\label{eqvmoderate}
\end{equation}
for some $C>0$. We denote by $\mathcal{P} (\mathbb{R}^d)$ the set of all weights $m$ which are $v$-moderate for some polynomial weight $v$.
\end{Definition}

\begin{Definition}\label{DefinitionModulationspace}
Given a window $g \in \mathcal{S} (\mathbb{R}^d) \backslash \left\{0 \right\}$, we define the short-time Fourier transform of $f \in  \mathcal{S} (\mathbb{R}^d)$ by
\begin{equation}
  V_gf(x,\omega)=\left( f,\pi(x,\omega)g\right)_{L^2 (\mathbb{R}^d)}=\int_{\mathbb{R}^d} f(t)\overline{g(t-x)}e^{-2\pi it \cdot \omega}dt.
\label{eqModulationspace1}
\end{equation}

This extends to $f \in  \mathcal{S}^{\prime} (\mathbb{R}^d) $, if we use the duality bracket:
\begin{equation}
V_gf(x,\omega)=\langle f,\pi(x,\omega) \overline{g} \rangle.
\label{eqModulationspace2}
\end{equation}
To study the time-frequency content of a function or tempered distribution, we shall consider the mixed norms:
\begin{equation}
||F||_{L_{x,\omega}^{r,s} (\mathbb{R}^{2d})}=\left(\int_{\mathbb{R}^d} \left(\int_{\mathbb{R}^d} |F(x, \omega)|^r d x\right)^{\frac{s}{r}} d \omega\right)^{\frac{1}{s}},
\label{eqLieb4.4}
\end{equation}
and
\begin{equation}
||F||_{L_{\omega,x}^{r,s} (\mathbb{R}^{2d})}=\left(\int_{\mathbb{R}^d} \left(\int_{\mathbb{R}^d} |F(x, \omega)|^r d \omega \right)^{\frac{s}{r}} d x\right)^{\frac{1}{s}},
\label{eqLieb4.5}
\end{equation}
for $F \in \mathcal{S}^{\prime} (\mathbb{R}^{2d})$ and $1 \le r,s< \infty$, with the obvious modification for $r$ or $s= \infty$.

Given a weight $m \in \mathcal{P} (\mathbb{R}^{2d})$, the modulation space $M_m^{r,s} (\mathbb{R}^d)$ is defined as the set of all $f \in \mathcal{S}^{\prime} (\mathbb{R}^{d})$ such that
\begin{equation}
||f||_{M_m^{r,s}  (\mathbb{R}^d)}= ||m V_gf||_{L_{x,\omega}^{r,s} (\mathbb{R}^{2d})} = \left(\int_{\mathbb{R}^d} \left(\int_{\mathbb{R}^d} |V_g f(x, \omega) m (x, \omega) |^r d x\right)^{\frac{s}{r}} d \omega\right)^{\frac{1}{s}} < \infty.
\label{eqModulationspace3}
\end{equation}
We shall write $M_m^r$, when $r=s$ and $M^{r,s}$, when $m\equiv 1$.
\end{Definition}

If instead of the order of integration (\ref{eqLieb4.4}), we use the alternative order (\ref{eqLieb4.5}), we obtain the so-called Wiener amalgam spaces $\mathbf{W}\left(\mathcal{F} \mathbf{L}^p, \ell ^q\right)(\mathbb{R}^d)$, which were introduced in \cite{Feichtinger2A}. The Fourier properties of these spaces for separable weights are explained in \cite{Feichtinger4}.

Among the modulation spaces we find for instance $L^2 (\mathbb{R}^d)$, or the Sobolev spaces $H^s (\mathbb{R}^d) $, the Feichtinger algebra ${\bf S}_0(\mathbb{R}^d)=M^{1,1} (\mathbb{R}^d)$ \cite{Feichtinger1}, but also the spaces $M_{\alpha, \beta}^{r,s} (\mathbb{R}^d)$ mentioned in the Introduction. Indeed, for $m(x, \omega) = \langle x \rangle^{\alpha} \langle \omega \rangle^{\beta}$, then $M_m^{r,s} (\mathbb{R}^d)= M_{\alpha, \beta}^{r,s} (\mathbb{R}^d)$ as in (\ref{eqModulationSpace1}). For the particular choice of weights $v_s(x,\omega)= \left(1+|x|^2+|\omega|^2 \right)^{s/2}$, and $r=s=2$ we obtain the so-called Shubin class spaces $\left(\mathbf{Q}_s (\mathbb{R}^d), \| \cdot \|_{\mathbf{Q}_s}\right)$. Notice that these spaces are denoted by $M_{v_s}^2$ in \cite{Grochenig}. Special versions of the form $M_m^2$ are also given in \cite{Aceska}.

For future reference, we state the following Proposition (Proposition 11.3.1 in \cite{Grochenig}):

\begin{Proposition}\label{PropositionModulationPartial}
Let $g \in \mathcal{S} (\mathbb{R}^d) \backslash \left\{0 \right\}$.
\begin{enumerate}
\item If $m(x, \omega) =m(x)$, then $M_m^2=L_m^2$.

\vspace{0.2 cm}
\item If $m(x, \omega) =m(\omega)$, then $M_m^2= \mathcal{F} L_m^2$.
\end{enumerate}
\end{Proposition}

\section{Lieb's uncertainty principle for Modulation spaces}

The radar ambiguity function is given by:
\begin{equation}
A(f,g) (x, \omega) = \int_{\mathbb{R}^d} f \left(t - \frac{x}{2} \right) \overline{g \left(t + \frac{x}{2} \right)} e^{-2 \pi i \omega \cdot t} dt.
\label{eqAmbiguityFunction1}
\end{equation}

Let us remark that the ambiguity function and its Fourier transform - the Wigner transform - can be extended to $\left({\bf S}_0^{\prime} (\mathbb{R}^d), \|\cdot \|_{{\bf S}_0^{\prime}}\right)$, since it is evaluated by taking a tensor product, then an automorphism of phase space $\mathbb{R}^{2d}$ and then a partial Fourier transform. All these are well defined operators for the Feichtinger algebra $\left({\bf S}_0(\mathbb{R}^d), \|\cdot\|_{{\bf S}_0}\right)$, and extend, by duality, to ${\bf S}_0^{\prime} (\mathbb{R}^d)$ in a unique way (see \cite{Feichtinger6,Feichtinger7,Hormann}).

It is straightforward to obtain the following properties.
\begin{equation}
A(f,g) (x, \omega) = A(\widehat{f},\widehat{g}) (-\omega,x).
\label{eqAmbiguityFunction2}
\end{equation}

From (\ref{eqAmbiguityFunction1}) it follows that
\begin{equation}
A(f,g) (x, \omega)= e^{-i \pi \omega \cdot x} V_g f (-x, \omega)~.
\label{eqAmbiguityFunction3}
\end{equation}

Lieb proved (Theorem 1 in \cite{Lieb}) the following theorem:

\begin{Theorem}[Lieb]\label{Lieb2}
Let $r >2$. If $f \in L^p(\mathbb{R})$ and $g \in L^{p^{\prime}} (\mathbb{R})$ for $r^{\prime} \le p , p^{\prime} \leq r$, then:
\begin{equation}
||A(f,g)||_{L^r (\mathbb{R}^2)} \le \left[H(r,p)\right]^{1/r} ||f||_{L^p (\mathbb{R})} ||g||_{L^{p^{\prime}} (\mathbb{R})},
\label{eqLieb01}
\end{equation}
where $H(r,p) >0$ is given by:
\begin{equation}
\left[H(r,p) \right]^2 = \frac{p p^{\prime}}{r^2} |r-2|^{2-r} |r-p|^{-1+r/p} |r-p^{\prime}|^{-1+r/ p^{\prime}},
\label{eqLieb02}
\end{equation}
with the convention $0^0 \equiv 1$ when $r=p$ or $r= p^{\prime}$.
\end{Theorem}
This inequality is sometimes called {\it Lieb's uncertainty principle}, as it implies a lower bound on the measure of the support of $A(f,g)$ \cite{Grochenig}.

Lieb also proved a reverse inequality (Theorem 2 in \cite{Lieb}).

\begin{Theorem}[Lieb]\label{Lieb2}
Assume that, for a.e. fixed $x \in \mathbb{R}$, the function $t \mapsto f\left (t-\frac{x}{2} \right) \overline{g\left (t+\frac{x}{2} \right)}$ is in $L^1 (\mathbb{R})$ for almost every $x \in \mathbb{R}$. Let $1 \le r<2$ and assume that $0< ||A(f,g)||_{L^r (\mathbb{R}^2)}< \infty$. Then $f,g \in L^u (\mathbb{R})$ for every $r \le u \le r^{\prime}$. Moreover, for any $p$ with $r \le p, p^{\prime} \le r^{\prime}$, we have that:
\begin{equation}
||A(f,g)||_{L^r (\mathbb{R}^2)} \ge \left[H(r,p) \right]^{1/r} ||f||_{L^p (\mathbb{R})} ||g||_{L^{p^{\prime}} (\mathbb{R})}.
\label{eqLieb1}
\end{equation}
\end{Theorem}

\begin{Remark}\label{RemarkLieb1}
The constant $H(r,p)$ can be expressed as:
\begin{equation}
H(r,p) = C_{r^{\prime}}^r \left( \frac{C_{\frac{p}{r^{\prime}}}C_{\frac{p^{\prime}}{r^{\prime}}}}{C_{\frac{r}{r^{\prime}}}}\right)^{r/r^{\prime}},
\label{eqLieb4.1}
\end{equation}
where, for $0 < p \le \infty$, $C_p$ is the Babenko-Beckner \cite{Babenko,Beckner} constant, which reads:
\begin{equation}
C_p=\sqrt{\frac{p^{1/p}}{|p^{\prime}|^{1/p^{\prime}}}},
\label{eqLieb4.2}
\end{equation}
for $p \neq 1$ and $p \neq \infty$, and
\begin{equation}
C_1=C_{\infty}=1.
\label{eqLieb4.3}
\end{equation}
\end{Remark}

\begin{Remark}\label{RemarkLieb2}
The constants in (\ref{eqLieb01},\ref{eqLieb1}) are sharp. Lieb also proved that one can obtain an equality in (\ref{eqLieb01}) for $p,p^{\prime} > r^{\prime}$ and in (\ref{eqLieb1}) for $1<r<2$ if and only if $f,g$ are certain matched Gaussians.
\end{Remark}

The purpose of this section is to generalize inequality (\ref{eqLieb1}) to arbitrary dimension $d>1$ and for the mixed norms $|| \cdot ||_{L_{x,\omega}^{r,s} (\mathbb{R}^{2d})}$ and $||\cdot||_{L_{\omega,x}^{r,s} (\mathbb{R}^{2d})}$ (cf.(\ref{eqLieb4.4},\ref{eqLieb4.5})).

\begin{Remark}\label{RemarkLieb3}
A generalization of inequality (\ref{eqLieb01}) for modulation spaces was derived in \cite{Cordero}.
\end{Remark}

We shall prove the following generalization of the reverse Lieb inequality:

\begin{Theorem}\label{Bounds}
Assume that $t \mapsto \widehat{f} \left(t + \frac{\omega}{2} \right) \overline{\widehat{g} \left(t - \frac{\omega}{2} \right)}$ is in $L^1 (\mathbb{R}^d)$ for a.e. $\omega \in \mathbb{R}^d$. Notice that under these circumstances the definition of $A(f,g)$ makes sense, in view of (\ref{eqAmbiguityFunction2}). Let $1  \le r,s \leq 2$ and assume that $0 < ||A(f,g)||_{L_{x,\omega}^{r,s}(\mathbb{R}^{2d})}< \infty$. Then $\widehat{f} , \widehat{g} \in L^u (\mathbb{R}^d)$ for every $0 < u \le r^{\prime}$. Moreover, for every pair $u,v$ such that $0 <u,v \le r^{\prime}$ and
\begin{equation}
\frac{1}{u}+\frac{1}{v} =   \frac{1}{s} + \frac{1}{r^{\prime}},
\label{eqBounds1}
\end{equation}
we have:
\begin{equation}
||A(f,g)||_{L_{x,\omega}^{r,s} (\mathbb{R}^{2d})} \ge B(r,s,u,v) ||\widehat{f}||_{L^u (\mathbb{R}^d)} ||\widehat{g}||_{L^v (\mathbb{R}^d)},
\label{eqBounds2}
\end{equation}
where
\begin{equation}
B(r,s,u,v) =C_{r^{\prime}}^d \left( \frac{C_{u/r^{\prime}} C_{v/r^{\prime}}}{C_{s/r^{\prime}}} \right)^{d/ r^{\prime}}.
\label{eqBounds3}
\end{equation}
The constant in (\ref{eqBounds2}) is sharp and for $1<r<2$, $0<u,v < r^{\prime}$, we have an equality if and only $\widehat{f},\widehat{g}$ are Gaussians of the form:
\begin{equation}
\begin{array}{l}
\widehat{f} (\omega) = \exp \left[ - \omega \cdot \left(|m^{\prime}| A + i B \right) \omega  + c \cdot \omega + \gamma \right] \\
\\
\widehat{g} (\omega) = \exp \left[ - \omega \cdot \left(|n^{\prime}| A - i B \right) \omega  + \widetilde{c} \cdot \omega +  \widetilde{\gamma} \right]
\end{array}
\label{eqBounds4}
\end{equation}
where $A$ is a real, symmetric, positive-definite $d \times d$ matrix, $B$ is a real, symmetric $d \times d$ matrix, $c,\widetilde{c} \in \mathbb{C}^d$, $\gamma, \widetilde{\gamma} \in \mathbb{C}$, and:
\begin{equation}
m^{\prime}= \frac{u(r-1)}{ur-u-r}, \hspace{1 cm} n^{\prime} =\frac{v(r-1)}{vr-v-r}.
\label{eqBounds5}
\end{equation}
Here $m^{\prime}$ and $n^{\prime}$ are the H\"older duals of $m=\frac{u}{r^{\prime}}$ and $n=\frac{v}{r^{\prime}}$, respectively.

\end{Theorem}

From this theorem and its proof we obtain the following two corollaries:

\begin{Corollary}\label{CorollaryBoundsA}
Assume that $t \mapsto f \left(t - \frac{x}{2} \right) \overline{g \left(t + \frac{x}{2} \right)}$ is in $L^1 (\mathbb{R}^d)$ for a.e. $x \in \mathbb{R}^d$. Let $1  \le r,s \leq 2$ and assume that $0 < ||A(f,g)||_{L_{\omega,x}^{r,s}(\mathbb{R}^{2d})}< \infty$. Then $f ,g \in L^u (\mathbb{R}^d)$ for every $0 < u \le r^{\prime}$. Moreover, for every pair $u,v$ such that $0 <u,v \le r^{\prime}$ and (\ref{eqBounds1}) holds, we have:
\begin{equation}
||A(f,g)||_{L_{\omega,x}^{r,s} (\mathbb{R}^{2d})} \ge B(r,s,u,v) ||f||_{L^u (\mathbb{R}^d)} ||g||_{L^v (\mathbb{R}^d)},
\label{eqCorollaryBoundsA2}
\end{equation}
where
\begin{equation}
B(r,s,u,v) =C_{r^{\prime}}^d \left( \frac{C_{u/r^{\prime}} C_{v/r^{\prime}}}{C_{s/r^{\prime}}} \right)^{d/ r^{\prime}}.
\label{eqCorollaryBoundsA3}
\end{equation}
The constant in (\ref{eqCorollaryBoundsA2}) is sharp and for $1 < s < 2$ and $0<u,v < r^{\prime}$, we have an equality if and only $f,g$ are matched Gaussians of the form (\ref{eqBounds4},\ref{eqBounds5}) with $\widehat{f},\widehat{g}$ replaced by $f,g$, respectively and $\omega$ replaced by $x$.
\end{Corollary}

The proof of this corollary follows exactly the same steps of the proof of Theorem \ref{Bounds} (see section 3.2).

\begin{Corollary}\label{CorollaryBounds}
Let $1\leq r,s \leq 2$; $ 1 < p, q < \infty$ and $g \in \mathcal{S}(\mathbb{R}^d) \backslash \left\{0 \right\}$. Assume that $f$ is such that $0 < ||f||_{g, M^{r,s} (\mathbb{R}^d)} < \infty$ and $||f||_{L^p (\mathbb{R}^d)} < \infty$, $||\widehat{f}||_{L^q (\mathbb{R}^d)} < \infty$. Then $f, \widehat{f} \in L^u (\mathbb{R}^d)$ for every $0 < u \le r^{\prime}$. Moreover, for every $0 < v \le r^{\prime}$, such that (\ref{eqBounds1}) holds, we have:
\begin{equation}
||\widehat{f}||_{L^u (\mathbb{R}^d)} \le \frac{||f||_{g, M^{r,s} (\mathbb{R}^d)}}{B(r,s,u,v) || \widehat{g}||_{L^v (\mathbb{R}^d)}}
\label{eqCorollaryBounds2}
\end{equation}
and
\begin{equation}
||f||_{L^u (\mathbb{R}^d)} \le \frac{||f||_{g, M ^{r,s} (\mathbb{R}^d)}}{B(r,s,u,v) || g||_{L^v (\mathbb{R}^d)}}
\label{eqCorollaryBounds2.1}
\end{equation}

where $B(r,s,u,v)$ is given by (\ref{eqBounds3}).
\end{Corollary}

\begin{proof}
By assumption, we have $\widehat{f} \in L^q (\mathbb{R}^d)$ and $\widehat{g} \in L^{q^{\prime}} (\mathbb{R}^d)$. It follows by H\"older's inequality that the function $t \mapsto \widehat{f} \left( t + \frac{\omega}{2} \right) \overline{\widehat{g} \left( t - \frac{\omega}{2} \right)}$ is in $L^1 (\mathbb{R}^d)$ for a.e. $\omega \in \mathbb{R}^d$. Also, from (\ref{eqAmbiguityFunction3}) we have
\begin{equation}
\begin{array}{c}
\infty > ||f||_{g, M^{r,s} (\mathbb{R}^d)} = \\
\\
= \left(\int_{\mathbb{R}^d}  \left( \int_{\mathbb{R}^d} | V_g f(x, \omega)|^r  dx \right)^{s/r}~  d \omega \right)^{1/s} = \\
\\
=  ||V_g f||_{L_{x,\omega}^{r,s} (\mathbb{R}^{2d})}  = ||A(f,g)||_{L_{x, \omega}^{r,s} (\mathbb{R}^{2d})}
\end{array}
\label{eqCorollaryBounds3}
\end{equation}
By Theorem \ref{Bounds} we obtain (\ref{eqCorollaryBounds2}). Eq. (\ref{eqCorollaryBounds2.1}) is an immediate consequence of Corollary \ref{CorollaryBoundsA} and the fact that (cf.(\ref{eqAmbiguityFunction2})):
\begin{equation}
||A(f,g)||_{L_{x, \omega}^{r,s} (\mathbb{R}^{2d})} = ||A(\widehat{f},\widehat{g})||_{L_{\omega,x}^{r,s} (\mathbb{R}^{2d})} .
\label{eqCorollaryBounds4}
\end{equation}
\end{proof}

To prove Theorem \ref{Bounds} we shall need to recapitulate some classical theorems in harmonic analysis.

\subsection{Some results in harmonic analysis}

\vspace{0.3 cm}
\noindent
The first result is the Hausdorff-Young inequality:

\begin{Theorem}[Hausdorff-Young]\label{TheoremHausdorffYoung}
Let $r \in \left[1,2 \right]$. If $f \in L^{r} (\mathbb{R}^d)$, then  $\widehat{f} \in L^{r^{\prime}} (\mathbb{R}^d)$, and:
\begin{equation}
\|\widehat{f}\|_{L^{r^{\prime}} (\mathbb{R}^d)} \leq C_{r}^d \|f\|_{L^{r} (\mathbb{R}^d)},
\label{eqHausdorffYoung1}
\end{equation}
where $C_r$ is the Babenko-Beckner constant (\ref{eqLieb4.2},\ref{eqLieb4.3}). There is obviously a converse result. Suppose that $\widehat{f}$ exists and $\widehat{f} \in L^r (\mathbb{R}^d)$. Since $\left(\widehat{f}\right)^{\widehat{}}(x) = f(-x)$ then $f \in L^{r^{\prime}} (\mathbb{R}^d)$ and
\begin{equation}
\|f\|_{L^{r^{\prime}} (\mathbb{R}^d)} \leq C_{r}^d \|\widehat{f}\|_{L^{r} (\mathbb{R}^d)}.
\label{eqHausdorffYoung2}
\end{equation}
Equality is achieved in (\ref{eqHausdorffYoung1}) and in (\ref{eqHausdorffYoung2}) when $1 < r <2$ if and only if $f$ is a Gaussian of the form
\begin{equation}
f(x)=\gamma \exp \left( - x\cdot A x + c \cdot x \right),
\label{eqHausdorffYoung3}
\end{equation}
with $ \gamma \in \mathbb{C}$, $A$ is any real, symmetric, positive-definite $d \times d$ matrix and $c \in \mathbb{C}^d$. For $r=1$ we can obtain an equality for many different functions. For $r=r^{\prime}=2$ an equality holds for all functions (Plancherel's Theorem).
\end{Theorem}

The sharp constant in (\ref{eqHausdorffYoung1}) is due to Beckner \cite{Beckner} and under some more restrictive conditions to Babenko \cite{Babenko}. The criterion for equality was obtained by Lieb \cite{Lieb2}.

The second inequality is due to Young. It applies to the convolution of two functions:
\begin{equation}
(f \star g)(x) = \int_{\mathbb{R}^d} f(x-y) g(y) dy.
\label{eqconvolution}
\end{equation}

\begin{Theorem}[Young]\label{TheoremYoung}
Let $1 \le m,n,r \le \infty$ with $\frac{1}{m}+ \frac{1}{n} = 1+ \frac{1}{r}$. If $f \in L^m (\mathbb{R}^d)$ and $g \in L^n (\mathbb{R}^d)$, then $f \star g \in L^r (\mathbb{R}^d)$ and
\begin{equation}
||f \star g||_{L^r (\mathbb{R}^d)} \le \left( \frac{C_m C_n}{C_r} \right)^d ||f||_{L^m (\mathbb{R}^d)} ||g||_{L^n (\mathbb{R}^d)}.
\label{eqYoung1}
\end{equation}
Here $C_m, C_r$, \textit{etc} denote the Babenko-Beckner constant as defined in (\ref{eqLieb4.2},\ref{eqLieb4.3}). Equality in (\ref{eqYoung1}) holds for $m,n >1$ if and only if
\begin{equation}
\begin{array}{l}
f(x)=\gamma \exp \left[ - m^{\prime} (x- \overline{x}) \cdot A (x-\overline{x}) - i k \cdot x \right]\\
\\
g(x)=\widetilde{\gamma} \exp \left[ - n^{\prime} (x- \tilde{x}) \cdot A (x-\tilde{x}) + i k \cdot x \right]
\end{array}
\label{eqYoung2}
\end{equation}
where $\gamma, \widetilde{\gamma} \in \mathbb{C}$, $\overline{x},\tilde{x} ,k \in \mathbb{R}^d$ and $A$ is any real, symmetric, positive-definite $d \times d$ matrix.
\end{Theorem}

The sharp constant in Young's inequality was obtained independently by Beckner \cite{Beckner} and by Brascamp and Lieb \cite{Lieb3}.

The reverse inequality was first obtained by Leindler \cite{Leindler}. Its sharp version was derived by Brascamp and Lieb \cite{Lieb3}. A simple proof can be found in \cite{Barthe}.

\begin{Theorem}[Leindler]\label{TheoremLeindler}
Let $f$ and $g$ be non-negative, real-valued functions on $\mathbb{R}^d$ that are not identically zero and assume that $f \star g \in L^r  (\mathbb{R}^d)$ for $0 < r \le 1$. Let $0 <m,n \le 1$ be such that $\frac{1}{m}+\frac{1}{n}= 1+ \frac{1}{r}$. Then $ f\in L^m  (\mathbb{R}^d)$ and $ f\in L^n  (\mathbb{R}^d)$, and we have:
\begin{equation}
||f\star g||_{L^r  (\mathbb{R}^d)} \ge \left( \frac{C_m C_n}{C_r} \right)^d ||f||_{L^m  (\mathbb{R}^d)} ||g||_{L^n  (\mathbb{R}^d)}.
\label{eqLeindler1}
\end{equation}
Equality holds in (\ref{eqLeindler1}) when $0 <m,n<1$, if and only if
\begin{equation}
\begin{array}{l}
f(x)=\gamma \exp \left[ - |m^{\prime}| (x- \overline{x}) \cdot A (x- \overline{x})  \right]\\
\\
g(x)=\widetilde{\gamma} \exp \left[ - |n^{\prime}| (x- \tilde{x}) \cdot A (x-\tilde{x})  \right]
\end{array}
\label{eqLeindler2}
\end{equation}
where $\gamma, \widetilde{\gamma} \in \mathbb{R}^+$, $ \overline{x},\tilde{x}  \in \mathbb{R}^d$ and $A$ is any real, symmetric, positive-definite $d \times d$ matrix.
\end{Theorem}

Finally, we shall also use the following extension of Cauchy's functional equation to quadratics:

\begin{Theorem}[Lieb]\label{TheoremCauchy}
Let $\xi, \eta$ be complex-valued, Lebesgue measurable functions on $\mathbb{R}^d$ that satisfy $|\xi (t)|=| \eta(t)|=1$, for a.e. $t \in \mathbb{R}^d$. Suppose there exist functions $\mu: \mathbb{R}^d \to \mathbb{R}^d$ and $\nu: \mathbb{R}^d \to \mathbb{R}$ (which are not a priori measurable) such that for a.e. $\omega \in \mathbb{R}^d$ the following holds for a.e. $t \in \mathbb{R}^d$:
\begin{equation}
\xi \left(t + \frac{\omega}{2}\right) \eta \left(t - \frac{\omega}{2}\right) = \exp \left[ i \mu (\omega) \cdot t + i \nu (\omega)\right].
\label{eqCauchy1}
\end{equation}
Then there exist a real, symmetric $d \times d$ matrix $A$, $\overline{\xi}, \overline{\eta} \in \mathbb{R}^d$ and $\gamma, \delta \in \mathbb{R}$ such that:
\begin{equation}
\begin{array}{l}
\xi (t)= \exp \left( i t \cdot A t + i \overline{\xi} \cdot t + i \gamma \right)\\
\\
\eta (t)= \exp \left( - i t \cdot A t + i \overline{\eta} \cdot t + i \delta \right)
\end{array}
\label{eqCauchy2}
\end{equation}
Conversely, if $\xi$, $\eta$ are of the form (\ref{eqCauchy2}), then (\ref{eqCauchy1}) holds.
\end{Theorem}

The proof of this theorem is basically the same as that of Lemma 4 in \cite{Lieb} with the obvious adaptations to $d>1$.

\subsection{Proof of Theorem \ref{Bounds}}

\vspace{0.3 cm}
\noindent
We are now in a position to prove Theorem \ref{Bounds}. The proof follows essentially the same steps as in the proof of Theorem 2 of \cite{Lieb}.

Since by assumption, for a.e. fixed $\omega \in \mathbb{R}^d$, the function $t \mapsto \widehat{f} \left(t+ \frac{\omega}{2} \right) \overline{\widehat{g} \left(t- \frac{\omega}{2} \right)}$ is in $L^1 (\mathbb{R}^d)$, then by (\ref{eqAmbiguityFunction1},\ref{eqAmbiguityFunction2}) we conclude that $A(f,g)$ is the Fourier transform of this $L^1 $ function. We can thus use (\ref{eqHausdorffYoung2}) in Theorem \ref{TheoremHausdorffYoung} (with $p=1$) to obtain:
\begin{equation}
\int_{\mathbb{R}^d} |A (f,g) (x, \omega)|^r dx \ge C_{r^{\prime}}^{dr} \left( \int_{\mathbb{R}^d} |\widehat{f} \left(t+ \frac{\omega}{2} \right) \widehat{g} \left(t- \frac{\omega}{2} \right)|^{r^{\prime}} dt \right)^{\frac{r}{r^{\prime}}}
\label{proofConverseLieb1}
\end{equation}
for a.e. $ x \in \mathbb{R}^d$. Notice that the left-hand side of (\ref{proofConverseLieb1}) is finite for a.e. $ x \in \mathbb{R}^d$, since by assumption $A(f,g) \in L_{x, \omega}^{r,s} (\mathbb{R}^{2d})$. The integral on the right-hand side can be written as:
\begin{equation}
J(\omega)= \int_{\mathbb{R}^d} |\widehat{f} \left(t+ \frac{\omega}{2} \right)|^{r^{\prime}} ~| \widehat{g} \left(t- \frac{\omega}{2} \right)|^{r^{\prime}} dt = \left(| (\widehat{ f})^{\vee}|^{r^{\prime}}  \star | \widehat{ g} |^{r^{\prime}}\right) (- \omega),
\label{proofConverseLieb2}
\end{equation}
where $(\widehat{ f})^{\vee} (\xi) = \widehat{ f} (- \xi)$.

It follows from (\ref{proofConverseLieb1},\ref{proofConverseLieb2}):
\begin{equation}
||A(f,g)||_{L_{x, \omega}^{r,s} (\mathbb{R}^{2d})} \ge C_{r^{\prime}}^d \left(\int_{\mathbb{R}^d} |J(\omega)|^{\frac{s}{r^{\prime}}} d \omega \right)^{\frac{1}{s}} = C_{r^{\prime}}^d \| ~ | (\widehat{ f})^{\vee}|^{r^{\prime}}  \star | \widehat{ g} |^{r^{\prime}} \|_{L^{s/r^{\prime}} (\mathbb{R}^d)}^{\frac{1}{r^{\prime}}}
\label{proofConverseLieb3}
\end{equation}
From the Leindler inequality, we thus get:
\begin{equation}
\begin{array}{c}
||A(f,g)||_{L_{x, \omega}^{r,s} (\mathbb{R}^{2d})}  \ge C_{r^{\prime}}^d \left(\frac{C_m C_n}{C_{s/ r^{\prime}}} \right)^{\frac{d}{r^{\prime}}}  \| ~ | \widehat{ f}|^{r^{\prime}} \|_{L^m (\mathbb{R}^d)}^{\frac{1}{r^{\prime}}} ~\| ~ | \widehat{ g}|^{r^{\prime}} \|_{L^n (\mathbb{R}^d)}^{\frac{1}{r^{\prime}}} =\\
\\
=C_{r^{\prime}}^d \left(\frac{C_m C_n}{C_{s/ r^{\prime}}} \right)^{\frac{d}{r^{\prime}}}  \| ~  \widehat{ f} \|_{L^{m r^{\prime}} (\mathbb{R}^d)} ~\| ~  \widehat{ g}  \|_{L^{nr^{\prime}} (\mathbb{R}^d)}
\end{array}
\label{proofConverseLieb3}
\end{equation}
where $0< m,n \le 1$ with
\begin{equation}
\frac{1}{m} + \frac{1}{n}= 1+ \frac{r^{\prime}}{s} .
\label{proofConverseLieb4}
\end{equation}
We thus recover (\ref{eqBounds2}) with $0 < u=m r^{\prime} \le r^{\prime} $, $0 < v=n r^{\prime} \le r^{\prime} $ and $\frac{1}{u} +\frac{1}{v}= \frac{1}{r^{\prime}} + \frac{1}{s}$. Notice that if $r=s$, then $u,v$ are conjugate to each other, as in Lieb's inequality (\ref{eqLieb1}).

\vspace{0.3 cm}
\noindent
From Theorem \ref{TheoremHausdorffYoung}, we have an equality in (\ref{proofConverseLieb1}) for $1<r<2$, if and only if (cf.(\ref{eqHausdorffYoung3})):
\begin{equation}
\widehat{f} \left(t+ \frac{\omega}{2} \right) \overline{\widehat{g} \left(t- \frac{\omega}{2} \right)} = \lambda (\omega) \exp \left[- t\cdot C (\omega) t + c(\omega) \cdot t \right]
\label{proofConverseLieb5}
\end{equation}
for a.e. $\omega \in \mathbb{R}^d$, and where $\lambda (\omega) \in \mathbb{C}$, $ C (\omega)$ is any real, symmetric, positive-definite $d \times d$ matrix and $c(\omega)$ is any vector in $\mathbb{C}^d$.

On the other hand, we have an equality in (\ref{proofConverseLieb3}) for $0<u,v<r^{\prime}$ (or equivalently, for $0<n,m <1$) if and only if
\begin{equation}
\begin{array}{l}
|\widehat{f} (\omega)|= \exp \left[- |m^{\prime}| (\omega-\zeta) \cdot A (\omega-\zeta) + \delta \right] \\
\\
|\widehat{g} (\omega)|=  \exp \left[- |n^{\prime}| (\omega-\chi) \cdot A (\omega-\chi) +\widetilde{\delta}  \right]
\end{array}
\label{proofConverseLieb5}
\end{equation}
where $A$ is a real, symmetric positive-definite matrix, $\zeta,\chi \in \mathbb{R}^d$, $\delta, \widetilde{\delta} \in \mathbb{R}$, and:
\begin{equation}
\begin{array}{l}
m^{\prime}= \frac{m}{m-1}=\frac{u}{u-r^{\prime}}= \frac{u(r-1)}{u(r-1)-r}\\
\\
n^{\prime}= \frac{n}{n-1}=\frac{v}{v-r^{\prime}}= \frac{v(r-1)}{v(r-1)-r}
\end{array}
\label{proofConverseLieb5.1}
\end{equation}
Since these functions never vanish, we may safely define
\begin{equation}
\xi (\omega) \equiv \frac{\widehat{f} (\omega)}{|\widehat{f} (\omega)|}, \hspace{1 cm} \eta (\omega) \equiv \frac{\overline{\widehat{g} (\omega)}}{|\widehat{g} (\omega)|}.
\label{proofConverseLieb6}
\end{equation}
By a simple inspection, we conclude that $\xi$ and $\eta$ satisfy all the conditions stated in Theorem \ref{TheoremCauchy} for $\mu (\omega)= \mbox{Im}  \left(c(\omega)\right)$ and $\nu (\omega)= -i \ln \left( \lambda (\omega) / | \lambda (\omega)|\right)$. We conclude that:
\begin{equation}
\begin{array}{l}
\xi (\omega)= \exp \left(- i \omega \cdot B \omega + i g \cdot \omega + i \rho \right)\\
\\
\eta (\omega)= \exp \left( i \omega \cdot B \omega + i h \cdot \omega + i \widetilde{\rho} \right)
\end{array}
\label{proofConverseLieb7}
\end{equation}
where $B$ is any real, symmetric $d \times d$ matrix, $g,h \in \mathbb{R}^d$ and $\rho, \widetilde{\rho} \in \mathbb{R}$. From $\widehat{f} (\omega)=\xi (\omega)|\widehat{f} (\omega)|$ and $\widehat{g} (\omega)=\overline{\eta (\omega)} |\widehat{g} (\omega)|$, and (\ref{proofConverseLieb5},\ref{proofConverseLieb7}) we recover (\ref{eqBounds4}), which concludes the proof.

\section{Cowling-Price type uncertainty principle}

In the sequel we shall consider functionals of the form:
\begin{equation}
\mathfrak{F}_{a,b}^{p,q} [f]= \| ~|x|^a f \|_{L^p (\mathbb{R}^d)} + \| ~|\omega|^b \widehat{f} \|_{L^q (\mathbb{R}^d)}
\label{eqFunctional1}
\end{equation}

\begin{Proposition}\label{PropositionFunctional1}
Let $1 \le p,q \le \infty$ and $a,b \ge 0$. The functional $\mathfrak{F}_{a,b}^{p,q} $ is continuous and convex in $L_a^p (\mathbb{R}^d) \cap \mathcal{F} L_b^q (\mathbb{R}^d)$.
\end{Proposition}

\begin{proof}
Convexity is obvious. The rest is an immediate consequence of the fact that the functional is just the natural norm (in the spririt of H. Triebel \cite{Triebel}) of the function space $L_a^p (\mathbb{R}^d) \cap \mathcal{F} L_b^q (\mathbb{R}^d)$, and is therefore continuous.
\end{proof}

Before we continue let us make the following observation:
\begin{Lemma}\label{LemmaReflexivity}
Let $1 < p,q < \infty$ and $a,b \ge 0$. The space $L_a^p (\mathbb{R}^d) \cap \mathcal{F} L_b^q (\mathbb{R}^d)$ is reflexive.
\end{Lemma}

\begin{proof}
Since $L_a^p (\mathbb{R}^d)$ and $ \mathcal{F} L_b^q (\mathbb{R}^d)$ are reflexive, so is $\mathcal{G} \equiv L_a^p (\mathbb{R}^d) \times \mathcal{F} L_b^q (\mathbb{R}^d)$, when we use the usual product norm $\|(x,y)\|_{X \times Y}= \|x\|_X+ \|y\|_Y$ for the product of normed vector spaces $X\times Y$ and dual bracket $\langle (x^{\ast},y^{\ast}), (x,y) \rangle_{X \times Y}= \langle x^{\ast},x\rangle_X+ \langle y^{\ast},y \rangle_Y$, for $(x^{\ast} ,y^{\ast} )\in X^{\prime} \times Y^{\prime}, (x,y) \in X \times Y$.  Then the set $\mathcal{K} \equiv \left\{ (f,g) \in \mathcal{G}: ~ f=g \right\}$ is a closed subspace of $\mathcal{G}$ with respect to the previous product topology. Therefore, $\mathcal{K}$ is reflexive. Since $L_a^p (\mathbb{R}^d) \cap \mathcal{F} L_b^q (\mathbb{R}^d)$ is isomorphic to $\mathcal{K}$, it is also reflexive.
\end{proof}

Notice that this also follows from an old result concerning dual spaces of sums and intersections of Banach spaces (when both are embedded continuously into $\mathcal{S}^{\prime}(\mathbb{R}^d)$) \cite{Liu}.

To avoid a proliferation of indices, we shall use in the sequel the notation
\begin{equation}
\mathcal{B}_1 (\mathbb{R}^d)= L_a^p (\mathbb{R}^d) \cap \mathcal{F} L_b^q (\mathbb{R}^d), \hspace{1 cm}
\mathcal{B}_2 (\mathbb{R}^d)= M_{\alpha, \beta}^{r,s} (\mathbb{R}^d)
\label{eqWeakCompactness1}
\end{equation}
with indices $a,b,p,q,r,s,\alpha, \beta$ to be specified. Moreover, we shall assume a fixed window $g \in \mathcal{S} (\mathbb{R}^d) \backslash \left\{ 0 \right\}$ for the norm of the modulation space $M_{\alpha, \beta}^{r,s} (\mathbb{R}^d)$.

We denote by $\overline{B_1} (R)$ the closed ball of radius $R>0$ in $\mathcal{B}_1 (\mathbb{R}^d)$:
\begin{equation}
\overline{B_1} (R)= \left\{f \in \mathcal{B}_1 (\mathbb{R}^d): ~ ||f||_{\mathcal{B}_1 (\mathbb{R}^d)}  \le R \right\}.
\label{eqWeakCompactness2}
\end{equation}

\begin{Proposition}\label{PropositionWeakCompactness}
Let $0 <r,s \le 2$, $\alpha, \beta, a, b \ge 0$ and $1 < p,q < \infty$. Suppose that $ r \le p$ and $s \le q$, and that inequality (\ref{eqGalperinGrochenig1}) holds with all factors non-negative. If, for $R>0$, the set
\begin{equation}
\mathcal{U}(R) = \left\{f \in \overline{B_1} (R): || f||_{\mathcal{B}_2 (\mathbb{R}^d)}=1 \right\}
\label{eqWeakCompactness4}
\end{equation}
is nonempty, then it is a weakly sequentially compact subset of $\mathcal{B}_1 (\mathbb{R}^d)$.
\end{Proposition}

\begin{proof}
Suppose that $\mathcal{U}(R)$ is nonempty. Let $(f_n)_n$ be an arbitrary sequence in $\mathcal{U}(R)$. Since $ \mathcal{B}_1 (\mathbb{R}^d)$ is reflexive (Lemma \ref{LemmaReflexivity}), we conclude that $(f_n)_n$ has a weakly convergent subsequence, say $(g_k)_k$ with $g_k \rightharpoonup g$ for some $g \in \mathcal{B}_1 (\mathbb{R}^d)$. Since $ \overline{B_1} (R)$ is convex and closed, we have by Mazur's Theorem that $g \in  \overline{B_1} (R)$. It remains to prove that $||g||_{\mathcal{B}_2 (\mathbb{R}^d)}=1$. As $\mathcal{B}_1 (\mathbb{R}^d)$ is compactly embedded in $\mathcal{B}_2 (\mathbb{R}^d)$ (see Theorem \ref{TheoremGalperinGrocheing}), the sequence $(g_k)_k$ has a subsequence $(h_l)_l$ converging strongly in $\mathcal{B}_2 (\mathbb{R}^d)$ to some $h \in \mathcal{B}_2 (\mathbb{R}^d)$. By the continuity of the norm, we have: $||h||_{\mathcal{B}_2 (\mathbb{R}^d)} = \lim_{l \to \infty} ||h_l||_{\mathcal{B}_2 (\mathbb{R}^d)}=1$. We conclude the proof by showing that $h=g$ a.e.. Since $(h_l)_l$ is a subsequence of $(g_k)_k$, we have for all $u \in \mathcal{B}_2^{\prime} (\mathbb{R}^d)$:
\begin{equation}
\begin{array}{c}
\lim_{l \to \infty} \langle u, g-h_l \rangle_2 = \langle u,g \rangle_2- \lim_{l \to \infty} \langle u, h_l \rangle_2=\\
\\
=\langle u,g \rangle_1- \lim_{l \to \infty} \langle u, h_l \rangle_1 = \langle u,g \rangle_1- \langle u,g \rangle _1=0,
\end{array}
\label{eqWeakCompactness5}
\end{equation}
where $\langle \cdot, \cdot \rangle_1,~ \langle \cdot, \cdot \rangle_2$ denote the duality brackets in $\mathcal{B}_1 (\mathbb{R}^d),~\mathcal{B}_2 (\mathbb{R}^d)$, respectively, and where we used the fact that $u \in \mathcal{B}_2^{\prime} (\mathbb{R}^d) \subset \mathcal{B}_1^{\prime} (\mathbb{R}^d)$ and $g,h_l \in \mathcal{B}_1 (\mathbb{R}^d) \subset \mathcal{B}_2 (\mathbb{R}^d)$.

On the other hand, since $(h_l)_l$ converges to $h$ strongly in $\mathcal{B}_2 (\mathbb{R}^d)$, it also converges weakly, and thus:
\begin{equation}
\lim_{l \to \infty} \langle u,h_l \rangle_2 = \langle u,h \rangle_2.
\label{eqWeakCompactness6}
\end{equation}
From (\ref{eqWeakCompactness5},\ref{eqWeakCompactness6}), we conclude that
\begin{equation}
\langle u,g-h \rangle_2=0,
\label{eqWeakCompactness7}
\end{equation}
for all $u \in \mathcal{B}_2^{\prime} (\mathbb{R}^d)$, and hence $g=h$ a.e..
\end{proof}

\begin{Lemma}\label{LemmaWeakLowerSemicontinuity}
Assume the same conditions as in the previous Proposition. If $\mathcal{U} (R)$ is nonempty, then the functional $\mathfrak{F}_{a,b}^{p,q}$ is weakly lower semicontinuous in $\mathcal{U} (R)$.
\end{Lemma}

\begin{proof}
The closed ball $\overline{B_1} (R)$ is a convex and closed subset of $\mathcal{B}_1 (\mathbb{R}^d)$. Since $\mathfrak{F}_{a,b}^{p,q}$ is convex and continuous, it is weakly lower semicontinuous on $\overline{B_1} (R)$. On the other hand, $\mathcal{U} (R)$ is a weakly sequentially compact subset of $\mathcal{B}_1 (\mathbb{R}^d)$. Then the restriction of $\mathfrak{F}_{a,b}^{p,q}$ to $\mathcal{U} (R)$ is also weakly lower semicontinuous.
\end{proof}

\begin{Theorem}\label{TheoremMinimum}
Let $1<r,s \le 2$, $\alpha, \beta \ge 0$, $0< a,b \le r^{\prime}$ and $r\leq p \leq r^{\prime}$, $s\leq q \leq r^{\prime}$. If, in addition, inequality (\ref{eqGalperinGrochenig1}) holds, with all factors non-negative, then the functional
\begin{equation}
\mathfrak{F}_{a,b}^{p,q} \left[f \right] = \|~ |x|^a f \|_{L^p (\mathbb{R}^d)} + \|~ |\omega|^b \widehat{f} \|_{L^q (\mathbb{R}^d)}
\label{eqMinimum2}
\end{equation}
attains a minimum in
\begin{equation}
\Omega= \left\{ f \in L_a^p (\mathbb{R}^d) \cap \mathcal{F} L_b^q (\mathbb{R}^d): ~ ||f||_{M_{\alpha, \beta}^{r,s} (\mathbb{R}^d)}=1 \right\}.
\label{eqMinimum3}
\end{equation}

\end{Theorem}

\begin{proof}
Let $C_G >0$ denote the sharp constant in the Galperin-Gr\"ochenig inequality:
\begin{equation}
C_G= \mbox{inf} \frac{|| f||_{\mathcal{B}_1 (\mathbb{R}^d)}}{|| f||_{\mathcal{B}_2 (\mathbb{R}^d)}},
\label{eqMinimum4}
\end{equation}
where the infimum is taken over all $f \in \mathcal{B}_1 (\mathbb{R}^d) \backslash \left\{0 \right\} $. Consequently, there exists $f_1 \in \Omega$ such that
\begin{equation}
C_G < ||f_1||_{\mathcal{B}_1 (\mathbb{R}^d)} < 2 C_G.
\label{eqMinimum5}
\end{equation}
Let
\begin{equation}
M= \mbox{max} \left\{\frac{ap-1}{p} , \frac{bq-1}{q} ,0 \right\} ~,
\label{eqMinimum6}
\end{equation}
and
\begin{equation}
R_G= 2^M \left( \frac{1}{B(r,s,q,u) ||\widehat{g}||_{L^u (\mathbb{R}^d)}} + \frac{1}{B(r,s,p,v) ||g||_{L^v (\mathbb{R}^d)}} + 2 C_G \right) ~,
\label{eqMinimum7}
\end{equation}
where $0<u,v \le r^{\prime}$ are such that:
\begin{equation}
\frac{1}{q}+\frac{1}{u}= \frac{1}{r^{\prime}}+ \frac{1}{s} , \hspace{1 cm} \frac{1}{p}+\frac{1}{v}= \frac{1}{r^{\prime}}+ \frac{1}{s} .
\label{eqMinimum7.1}
\end{equation}

From (\ref{eqMinimum5}), we conclude that $ f_1 \in \overline{B_1} (R_G)$, and hence $\mathcal{U} (R_G)$, as defined in (\ref{eqWeakCompactness4}) is nonempty. By Proposition \ref{PropositionWeakCompactness} it is weakly sequentially compact and $\mathfrak{F}_{a,b}^{p,q}$ is weakly lower semicontinuous therein. Consequently, there exists a minimizer $f_0$ of $\mathfrak{F}_{a,b}^{p,q}$ in $\mathcal{U} (R_G)$. It remains to prove that $f_0$ is a minimizer everywhere in $\Omega$.

Suppose that $ap \ge 1$. We then have from the inequality $|x+y|^t \leq 2^{t-1} (|x|^t +|y|^t)$ for $t \ge 1$:
\begin{equation}
\begin{array}{c}
||f||_{L_a^p (\mathbb{R}^d)} = \left( \int_{\mathbb{R}^d} (1+ |x|)^{ap} |f(x)|^p dx \right)^{1/p} \\
\\
\le 2^{(ap-1)/p} \left( ||f||_{L^p (\mathbb{R}^d)}^p + \| ~|x|^a f \|_{L^p (\mathbb{R}^d)}^p \right)^{1/p} \le 2^M \left( ||f||_{L^p (\mathbb{R}^d)} + \| ~ |x|^a f \|_{L^p (\mathbb{R}^d)} \right)
\end{array}
\label{eqMinimum8}
\end{equation}
Alternatively, if $ap <1$, then we obtain again:
\begin{equation}
\begin{array}{c}
||f||_{L_a^p (\mathbb{R}^d)} = \left( \int_{\mathbb{R}^d} (1+ |x|)^{ap} |f(x)|^p dx \right)^{1/p}  \\
\\
 \le \left( ||f||_{L^p (\mathbb{R}^d)}^p + \| ~|x|^a f \|_{L^p (\mathbb{R}^d)}^p \right)^{1/p} \le 2^M \left( ||f||_{L^p (\mathbb{R}^d)} + \| ~ |x|^a f \|_{L^p (\mathbb{R}^d)} \right)
\end{array}
\label{eqMinimum9}
\end{equation}
In the same fashion, we have:
\begin{equation}
||\widehat{f}||_{L_b^q (\mathbb{R}^d)} \le 2^M \left( ||\widehat{f}||_{L^q (\mathbb{R}^d)} + \| ~ |\omega|^b \widehat{f} \|_{L^q (\mathbb{R}^d)} \right)
\label{eqMinimum10}
\end{equation}
Altogether, we conclude that:
\begin{equation}
\begin{array}{c}
||f||_{\mathcal{B}_1 (\mathbb{R}^d)}= ||f||_{L_a^p (\mathbb{R}^d)} + ||\widehat{f}||_{L_b^q (\mathbb{R}^d)} \\
\\
\le 2^M \left(||f||_{L^p (\mathbb{R}^d)}+ ||\widehat{f}||_{L^q (\mathbb{R}^d)} + \mathfrak{F}_{a,b}^{p,q} \left[f \right] \right)
\end{array}
\label{eqMinimum11}
\end{equation}
If $f$ is such that $||f||_{\mathcal{B}_1 (\mathbb{R}^d)} >  R_G$, then:
\begin{equation}
\begin{array}{c}
||f||_{L^p (\mathbb{R}^d)}+ ||\widehat{f}||_{L^q (\mathbb{R}^d)} + \mathfrak{F}_{a,b}^{p,q} \left[f \right]  \geq 2^{-M} \|f\|_{\mathcal{B}_1 (\mathbb{R}^d)}\\
\\
> \frac{1}{B(r,s,q,u)  ||\widehat{g}||_{L^u (\mathbb{R}^d)}} + \frac{1}{B(r,s,p,v)  ||g||_{L^v (\mathbb{R}^d)}} + 2 C_G
\end{array}
\label{eqMinimum12}
\end{equation}
But, in view of the trivial inequality $\|f\|_{g,M_{\alpha,\beta}^{r,s} (\mathbb{R}^d)} \geq \|f\|_{g,M^{r,s} (\mathbb{R}^d)}$ for all $\alpha, \beta \geq 0$ and of Corollary \ref{CorollaryBounds}, this entails:
\begin{equation}
\mathfrak{F}_{a,b}^{p,q} \left[f \right]  > 2C_G > ||f_1 ||_{\mathcal{B}_1 (\mathbb{R})} \ge \mathfrak{F}_{a,b}^{p,q} \left[f_1 \right] \ge\mathfrak{F}_{a,b}^{p,q} \left[f_0 \right]
\label{eqMinimum13}
\end{equation}
which concludes the proof.
\end{proof}

\section{An uncertainty principle in the Hilbert case}

In this section, we deal with the more specific case where $r=s=p=q=2$. All the functional spaces involved become Hilbert spaces. The advantage is that we can slightly generalize our construction and consider weights other than the powers $|x|^a$ and $|\omega|^b$. Moreover, as we shall see later, the minimizers are given by the eigenvalue equation of the inverse of a certain compact localization operator.

Let us start with the following definition.

\begin{Definition}\label{DefinitionWeight}
Given some weight $m_0 \in \mathcal{P} (\mathbb{R}^{2d})$, we call a pair of continuous functions $\psi, \phi \in L_{loc}^{\infty} (\mathbb{R}^d)$ $m_0$-admissible weights if
$$
m(x, \omega)= \sqrt{|m_0(x,\omega)|^2 + | \psi (x)|^2 + |\phi (\omega)|^2}
$$
satisfies $m \in \mathcal{P} (\mathbb{R}^{2d})$ and $\frac{m_0}{m} \in L_0^{\infty} (\mathbb{R}^{2d})$.
\end{Definition}

Let us remark that the conditions of the previous definition imply that $|\psi(x)| \to \infty$ and $|\phi(x)| \to \infty$ at a polynomial rate as $|x| \to \infty$.

A simple consequence of this definition and of the compactness theorem \ref{TheoremBoggiatto} is the following.

\begin{Corollary}\label{CorollaryEmbedding}
Let $(\psi, \phi)$ be $m_0$-admissible weights. Then we have the compact embedding
\begin{equation}
M_m^2 (\mathbb{R}^d) \subset \subset M_{m_0}^2 (\mathbb{R}^d).
\label{eqEmbedding1}
\end{equation}
\end{Corollary}

For future reference, we consider the following alternative norm.

\begin{Lemma}\label{LemmaPhiPsi1}
Let $(\psi,\phi)$ be $m_0$-admissible weights. Then the norm $\| \cdot \|_{M_m^2 (\mathbb{R}^d)}$ is equivalent to
\begin{equation}
||f||_{\psi,\phi,m_0}^2 = \|f\|_{M_{m_0}^2 (\mathbb{R}^d)}^2+ \| \psi f \|_{L^2 (\mathbb{R}^d)}^2 + \| \phi \widehat{f} \|_{L^2 (\mathbb{R}^d)}^2.
\label{eqPhisPsi2}
\end{equation}
\end{Lemma}

\begin{proof}
Following the same steps of the proof which leads to Proposition \ref{PropositionModulationPartial} (see the proof of Proposition 11.3.1 in \cite{Grochenig}), we obtain
\begin{equation}
\begin{array}{c}
\|f\|_{\psi,\phi,m_0}^2 = \|f\|_{M_{m_0}^2 (\mathbb{R}^d)}^2+ \| \psi f \|_{L^2 (\mathbb{R}^d)}^2 + \| \phi \widehat{f} \|_{L^2 (\mathbb{R}^d)}^2 \\
\\
\asymp
\|f\|_{M_{m_0}^2 (\mathbb{R}^d)}^2+  \|  f \|_{M_{|\psi|}^2 (\mathbb{R}^d)}^2 + \| f \|_{M_{|\phi|}^2 (\mathbb{R}^d)}^2=\\
\\
=\int_{\mathbb{R}^{2d}} m^2 (z) |V_g f (z)|^2 dz = \|f\|_{M_{m}^2 (\mathbb{R}^d)}^2.
\end{array}
\label{eqPhisPsi5}
\end{equation}
\end{proof}

Notice that we can associate the norm $\|\cdot \|_{\psi,\phi,m_0} $ to the inner product:
\begin{equation}
\begin{array}{c}
\left(u,v \right)_{\psi,\phi,m_0}= \int_{\mathbb{R}^{2d}} m_0(z)^2 V_gu (z) \overline{V_g v(z)} dz +\\
 \\
 +\int_{\mathbb{R}^d}  | \psi (x)|^2  u(x) \overline{v(x)} dx + \int_{\mathbb{R}^d}  | \phi (\omega)|^2 \widehat{u}(\omega) \overline{\widehat{v}(\omega)} d \omega~,
\end{array}
\label{eqPhisPsi3}
\end{equation}
such that $\left(u,u \right)_{\psi,\phi,m_0}= \|u\|_{\psi,\phi,m_0}^2$. Thus, the space $\mathcal{B}^{\psi,\phi,m_0} (\mathbb{R}^d)$ of measurable functions $f$ with finite $\|f\|_{\psi,\phi,m_0}< \infty$, endowed with the inner product $\left(\cdot,\cdot \right)_{\psi,\phi,m_0}$, is a Hilbert space and:
\[
\mathcal{B}^{\psi,\phi,m_0} (\mathbb{R}^d)= M_m^2 (\mathbb{R}^d)~.
\]

We shall now follow {\it mutatis mutandis} our proof of weak sequential compactness (Proposition \ref{PropositionWeakCompactness}) for this case. In the sequel we shall always assume that $(\psi, \phi)$ are a $m_0$-admissible set of weights with $m_0 \in \mathcal{P} (\mathbb{R}^{2d})$. Let us denote by $\overline{B}_R^{\psi,\phi,m_0}$ the following closed ball of radius $R>0$:
\begin{equation}
\overline{B}_R^{\psi,\phi,m_0}:= \left\{ f \in M_m^2 (\mathbb{R}^d): ~ \|f \|_{\psi, \phi,m_0} \le R \right\}.
\label{eqCompactnessHilbert1}
\end{equation}

\begin{Proposition}\label{PropositionCompactnessHilbert1}
Let $R>0$. If the set
\begin{equation}
U_R^{\psi,\phi,m_0}= \left\{ f \in \overline{B}_R^{\psi,\phi,m_0}: ~ \|f \|_{M_{m_0}^2(\mathbb{R}^d)} =1 \right\}
\label{eqCompactnessHilbert2}
\end{equation}
is nonempty, then it is a weakly sequentially compact subset of $M_m^2 (\mathbb{R}^d)$.
\end{Proposition}

\begin{proof}
Suppose that $U_R^{\psi,\phi,m_0}$ is nonempty. Let $(f_n)_n$ be an arbitrary sequence in $U_R^{\psi,\phi,m_0}$. Since $M_m^2 (\mathbb{R}^d)$ is reflexive (it is a Hilbert space), and the norms $\|\cdot \|_{M_m^2 (\mathbb{R}^d)}$ and $\|\cdot \|_{\psi,\phi,m_0}$ are equivalent, we conclude that $(f_n)_n$ has a weakly convergent subsequence $(g_k)_k$, say
\begin{equation}
g_k \rightharpoonup g,
\label{eqCompactnessHilbert3}
\end{equation}
for some $g \in M_m^2 (\mathbb{R}^d)$, and by Mazur's Theorem $g \in  \overline{B}_R^{\psi,\phi,m_0}$. It remains to prove that $\|g \|_{M_{m_0}^2 (\mathbb{R}^d)}=1$. From Corollary \ref{CorollaryEmbedding}, we conclude that the sequence $(g_k)_k$ has a subsequence $(h_l)_l$ converging strongly in $ M_{m_0}^2 (\mathbb{R}^d)$, say $\|h_l -h \|_{M_{m_0}^2(\mathbb{R}^d)} \to 0$, for some $h \in M_{m_0}^2(\mathbb{R}^d)$. By the continuity of the norm, we also have $\|h \|_{M_{m_0}^2(\mathbb{R}^d)} =1$. The proof is complete if we show that $g=h$.

Recall that $ M_{m_0}^2 (\mathbb{R}^d)$ is a Hilbert space with inner product
\begin{equation}
\left( u,v\right)_{M_{m_0}^2 (\mathbb{R}^d)}= \int_{\mathbb{R}^{2d}} m_0^2 (z) V_g u (z) \overline{V_g v(z)} dz.
\label{eqCompactnessHilbert4}
\end{equation}
The mapping $(u,v) \mapsto \left( u,v\right)_{M_{m_0}^2 (\mathbb{R}^d)}$ is a sesquilinear form on $\mathcal{B}^{\psi,\phi,m_0} (\mathbb{R}^d) \times \mathcal{B}^{\psi,\phi,m_0}(\mathbb{R}^d)=M_m^2(\mathbb{R}^d) \times M_m^2(\mathbb{R}^d)$. Moreover, it is bounded as we now prove. Since $m_0(z) \leq  m(z)$ for a.e. $z \in \mathbb{R}^{2d}$, we have from the Cauchy-Schwarz inequality and Lemma \ref{LemmaPhiPsi1}:
\begin{equation}
\begin{array}{c}
\left|\left( u,v\right)_{M_{m_0}^2 (\mathbb{R}^d)} \right| \leq \int_{\mathbb{R}^{2d}} m^2 (z)\left| V_g u (z) \overline{V_g v(z)} \right|dz\\
\\
\leq \|V_g u \|_{L_m^2 (\mathbb{R}^{2d})} ~ \|V_g v \|_{L_m^2 (\mathbb{R}^{2d})} = \|u \|_{M_m^2 (\mathbb{R}^d)} ~ \|v \|_{M_m^2 (\mathbb{R}^d)}\asymp   \|u \|_{\psi, \phi ,m_0} ~ \|v \|_{\psi, \phi ,m_0}~.
\end{array}
\label{eqCompactnessHilbert5}
\end{equation}
By the Riesz representation theorem, there exists a bounded linear operator $A:M_m^2 (\mathbb{R}^d) \to M_m^2 (\mathbb{R}^d)$, such that
\begin{equation}
\left( u,v\right)_{M_{m_0}^2 (\mathbb{R}^d)}=\left(A  u,v\right)_{\psi,\phi,m_0},
\label{eqCompactnessHilbert6}
\end{equation}
for all $u,v \in M_m^2(\mathbb{R}^d)$.

Let $u \in \mathcal{S} (\mathbb{R}^d)$. Since $m_0 \in \mathcal{P} (\mathbb{R}^{2d})$, we have that $\mathcal{S} (\mathbb{R}^d)$ is dense in $M_{m_0}^2 (\mathbb{R}^d)$. We then have:
\begin{equation}
\left( h_l-g,u\right)_{M_{m_0}^2 (\mathbb{R}^d)}=\left(A ( h_l-g), u\right)_{\psi,\phi,m_0}.
\label{eqCompactnessHilbert7}
\end{equation}
If we take the limit $l \to \infty$, the right-hand side of the previous equation vanishes, while the left-hand side becomes $\left( h-g,u\right)_{M_{m_0}^2 (\mathbb{R}^d)}$. Since $\mathcal{S} (\mathbb{R}^d)$ is dense in $M_{m_0}^2 (\mathbb{R}^d)$, we conclude that $h=g$ a.e.
\end{proof}

\begin{Remark}\label{RemarkToftGroch}
Since $\mathcal{B}^{\psi, \phi,m_0} (\mathbb{R}^d)$ may be identified with the modulation space $M^2_m(\mathbb{R}^d)$, we have a concrete realization of the operator $A$ in the preceding proof as a localization operator. Recall that a localization operator with symbol $\sigma$ and window $g\in \mathcal{S}(\mathbb{R}^d)$ is of the form
$$
A^g_{\sigma} f =\int_{\mathbb{R}^{2d}}\sigma(x,\omega)\langle f,\pi(x,\omega)g\rangle\pi(x,\omega)g\,dxd\omega.
$$
If we take for $\sigma$ the weight $m_0/m$, then $A^g_{m_0/m}$ is an isomorphism between $M^2_m(\mathbb{R}^d)$ and $M^2_{m_0}(\mathbb{R}^d)$ by the results in \cite{GrochenigToft1,GrochenigToft2}.
\end{Remark}

Next, we define the functional $\mathfrak{F}^{\psi,\phi,m_0}: M_m^2 (\mathbb{R}^d) \to \left[ \right. 0, \infty \left. \right)$,
\begin{equation}
\mathfrak{F}^{\psi,\phi,m_0} \left[f \right] = \|\psi f \|_{L^2 (\mathbb{R}^d)}^2 + \|\phi \widehat{f} \|_{L^2 (\mathbb{R}^d)}^2 = \|f \|_{\psi,\phi,m_0}^2 - \|f \|_{M_{m_0}^2 (\mathbb{R}^d)}^2.
\label{eqMinimizersHilbert1}
\end{equation}
We then have:
\begin{Lemma}\label{Lemmaweaklowersemicontinuous}
Let $R>0$ be such that $U_R^{\psi,\phi,m_0}$ as defined in Proposition \ref{PropositionCompactnessHilbert1} is nonempty. Then, the functional $\mathfrak{F}^{\psi,\phi,m_0}$ is weakly lower semicontinuous in $U_R^{\psi,\phi,m_0}$.
\end{Lemma}
\begin{proof}
Consider the maps
\begin{equation}
\left(\mathcal{M}_{\psi} f \right) (x):= \psi(x) f(x), \hspace{0.5 cm} \left(\mathcal{U}_{\phi} f\right) (\omega):= \left(\mathcal{M}_{\phi} \mathcal{F} f \right) (\omega)= \phi (\omega) \widehat{f} (\omega),
\label{eqOptimizers1}
\end{equation}
for every $x, \omega \in \mathbb{R}^d$ and $f \in \mathcal{S} (\mathbb{R}^d)$. They extend to bounded linear operators $M_m^2 (\mathbb{R}^{d}) \to L^2 (\mathbb{R}^d)$: $\|\mathcal{M}_{\psi} f\|_{L^2 (\mathbb{R}^d)}=\|\psi f\|_{L^2 (\mathbb{R}^d)} \leq \| f\|_{\psi,\phi,m_0}\asymp \| f\|_{M_m^2 (\mathbb{R}^d)}$. Thus, the map $f \mapsto \|\mathcal{M}_{\psi} f\|_{L^2(\mathbb{R}^d)}$ is a continuous and convex functional on the closed ball $\overline{B}_R^{\psi,\phi,m_0}$, which is a convex and closed subset of $M_m^2 (\mathbb{R}^{2d})$. This means that this map is weakly lower semicontinuous in $\overline{B}_R^{\psi,\phi,m_0}$.

On the other hand, by Proposition \ref{PropositionCompactnessHilbert1}, $U_R^{\psi,\phi,m_0}$ is a weakly sequentially compact subset of $\mathcal{B}^{\psi, \phi,m_0} (\mathbb{R}^d)$. Hence the restriction of $\|\mathcal{M}_{\psi} f\|_{L^2(\mathbb{R}^d)}$ to $U_R^{\psi,\phi,m_0} \subset \overline{B}_R^{\psi,\phi,m_0}$ is  weakly lower semicontinuous. The product of two nonnegative weakly lower semicontinuous functionals is again weakly lower semicontinuous, which entails that $\|\mathcal{M}_{\psi} f\|_{L^2(\mathbb{R}^d)}^2$ is weakly lower semicontinuous. The same can be said about $\|\mathcal{U}_{\phi} f\|_{L^2(\mathbb{R}^d)}^2$. Consequently, $\mathfrak{F}^{\psi,\phi,m_0}$, being the sum of these two functionals, is weakly lower semicontinuous in $U_R^{\psi,\phi,m_0}$.
\end{proof}
We next prove the existence of minimizers.
\begin{Theorem}\label{TheoremMinimizersHilbert}
Let $R>1$ be such that $U_R^{\psi,\phi,m_0}$ as defined in Proposition \ref{PropositionCompactnessHilbert1} is nonempty. Then there exists $f_0 \in U_R^{\psi,\phi,m_0}$ such that
\begin{equation}
\mathfrak{F}^{\psi,\phi,m_0} \left[f_0 \right]  \le \mathfrak{F}^{\psi,\phi,m_0} \left[f \right],
\label{eqMinimizersHilbert2}
\end{equation}
for all
\begin{equation}
f \in \Omega:= \left\{f \in M_m^2 (\mathbb{R}^d): ~ \|f\|_{M_{m_0}^2 (\mathbb{R}^d)} =1 \right\}.
\label{eqMinimizersHilbert3}
\end{equation}
\end{Theorem}
\begin{proof}
The set $U_R^{\psi,\phi,m_0}$ is weakly sequentially compact (cf. Proposition \ref{PropositionCompactnessHilbert1}). Moreover, the functional $\mathfrak{F}^{\psi,\phi,m_0}$ is weakly lower semicontinuous (Lemma \ref{Lemmaweaklowersemicontinuous}). Consequently, there exists a minimizer $f_0$ of $\mathfrak{F}^{\psi,\phi,m_0}$ in $U_R^{\psi,\phi,m_0}$. It remains to prove that $f_0$ is in fact a minimizer on the whole set $\Omega$.

From (\ref{eqMinimizersHilbert1}) and (\ref{eqMinimizersHilbert3}), we have:
\begin{equation}
\mathfrak{F}^{\psi,\phi,m_0} \left[f \right]= \|f\|_{\psi,\phi,m_0}^2 - 1,
\label{eqMinimizersHilbert4}
\end{equation}
for all $f \in \Omega$.

Since $f_0 \in U_R^{\psi,\phi,m_0}$ , it follows:
\begin{equation}
\mathfrak{F}^{\psi,\phi,m_0} \left[f_0 \right] \leq R^2- 1.
\label{eqMinimizersHilbert5}
\end{equation}
On the other hand, if $f \in \Omega \backslash U_R^{\psi,\phi,m_0}$:
\begin{equation}
\mathfrak{F}^{\psi,\phi,m_0} \left[f \right] > R^2- 1,
\label{eqMinimizersHilbert6}
\end{equation}
and the result follows.
\end{proof}

\section{Euler-Lagrange equations}

\subsection{The Banach case}

In this section, we shall derive an equation satisfied by the minimizer $f_0$ of the functional (\ref{eqMinimum2}) in the constraint set (\ref{eqMinimum3}). We thus consider the functional
\begin{equation}
\begin{array}{c}
\mathfrak{L}_{a,b}^{p,q} \left[f , \lambda \right]:= \mathfrak{F}_{a,b}^{p,q} \left[f \right] + \lambda \left(1- \|f \|_{M_{\alpha,\beta}^{r,s} (\mathbb{R}^d)} \right) =\\
 \\
 =\|~ |x|^a f \|_{L^p(\mathbb{R}^d)} + \| ~ |\omega|^b \widehat{f} \|_{L^q(\mathbb{R}^d)}  + \lambda \left(1- \|f \|_{M_{\alpha,\beta}^{r,s} (\mathbb{R}^d)} \right),
\end{array}
\label{eqBanachCase1}
\end{equation}
where $\lambda$ is a Lagrange multiplier.

\begin{Theorem}\label{TheoremBanachCase}
The minimizers $f_0$ of (\ref{eqMinimum2}) in the constraint set (\ref{eqMinimum3}) are solutions of the equation:
\begin{equation}
\begin{array}{c}
\|~ |x|^a f \|_{L^p (\mathbb{R}^d)}^{1-p} \int_{\mathbb{R}^d} |x|^{ap} |f(x)|^{p-2} f(x) \overline{u(x)} dx + \\
\\
+ \|~ |\omega|^b \widehat{f} \|_{L^q (\mathbb{R}^d)}^{1-q} \int_{\mathbb{R}^d} |\omega|^{bq} |\widehat{f}(\omega)|^{q-2} \widehat{f}( \omega) \overline{\widehat{u}( \omega)} d \omega =\\
\\
= \lambda \int_{\mathbb{R}^d} \|V_g f(\cdot, \omega) \|_{L_{\alpha}^r (\mathbb{R}^d)}^{s-r} \left[\int_{\mathbb{R}^d} |V_g f (x, \omega)|^{r-2} V_g f (x, \omega) \overline{V_g u(x, \omega)} \langle x \rangle^{\alpha r} dx \right] \langle \omega \rangle^{\beta s} d \omega,
\end{array}
\label{eqBanachCase1.1}
\end{equation}
for all $u \in L_a^p  (\mathbb{R}^d) \cap \mathcal{F} L_b^q  (\mathbb{R}^d)$, and where
\begin{equation}
 \|V_g f (\cdot, \omega) \|_{L_{\alpha}^r (\mathbb{R}^d)}= \left(\int_{\mathbb{R}^d} |V_g f(x, \omega)|^r \langle x \rangle^{\alpha r} dx \right)^{1/r}.
\label{eqBanachCase5}
\end{equation}
The minimizers satisfy this equation with the smallest possible value of $\lambda$.
\end{Theorem}

\begin{proof}
Let $u \in \mathcal{S} (\mathbb{R}^d)$. Recall that $\mathcal{S} (\mathbb{R}^d)$ is a dense subset of $L_a^p  (\mathbb{R}^d) \cap \mathcal{F} L_b^q  (\mathbb{R}^d)$. For $t \in \mathbb{R}$, we have:
\begin{equation}
\begin{array}{c}
\|~ |x|^a (f+t u) \|_{L^p(\mathbb{R}^d)}= \|~ |x|^a f \|_{L^p(\mathbb{R}^d)} +\\
 \\
 + t \|~ |x|^a f \|_{L^p(\mathbb{R}^d)}^{1-p} \int_{\mathbb{R}^d} |x|^{ap} |f(x)|^{p-2} \mathcal{R} \left(f(x) \overline{u(x)} \right) dx + \mathcal{O} (t^2),
\end{array}
\label{eqBanachCase2}
\end{equation}
where $ \mathcal{R}$ denotes the real part of a complex number.

Similarly:
\begin{equation}
\begin{array}{c}
\|~ |\omega|^b (\widehat{f}+t \widehat{u}) \|_{L^q(\mathbb{R}^d)}= \|~ |\omega|^b \widehat{f} \|_{L^q(\mathbb{R}^d)} +\\
 \\
 + t \|~ |\omega|^b \widehat{f} \|_{L^q(\mathbb{R}^d)}^{1-q} \int_{\mathbb{R}^d} |\omega|^{bq} |\widehat{f}(\omega)|^{q-2} \mathcal{R} \left(\widehat{f}(\omega) \overline{\widehat{u}(\omega)} \right) d\omega + \mathcal{O} (t^2).
\end{array}
\label{eqBanachCase3}
\end{equation}
Finally:
\begin{equation}
\begin{array}{c}
\|f+tu\|_{M_{\alpha, \beta}^{r,s} (\mathbb{R}^d)}= \|V_g(f+tu)\|_{L_{\alpha, \beta}^{r,s} (\mathbb{R}^{2d})}=\\
\\
= \|f\|_{M_{\alpha, \beta}^{r,s} (\mathbb{R}^d)} + t \|f\|_{M_{\alpha, \beta}^{r,s} (\mathbb{R}^d)}^{1-s} \times \\
\\
\times \int_{\mathbb{R}^d} \|V_g f (\cdot, \omega) \|_{L_{\alpha}^r (\mathbb{R}^d)}^{s-r} \left[\int_{\mathbb{R}^d} |V_g f (x, \omega)|^{r-2} \mathcal{R} \left( V_g f (x, \omega) \overline{V_g u (x, \omega)} \right) \langle x \rangle^{\alpha r} dx \right] \langle \omega \rangle^{\beta s} d \omega+ \\
\\
+ \mathcal{O} (t^2)
\end{array}
\label{eqBanachCase4}
\end{equation}
where $ \|V_g f (\cdot, \omega) \|_{L_{\alpha}^r (\mathbb{R}^d)}$ is as in eq.(\ref{eqBanachCase5}).

The stationarity condition for (\ref{eqBanachCase1}) is obtained from the Fr\'echet derivative:
\begin{equation}
\begin{array}{c}
0 = \lim_{t \to 0} \frac{1}{t} \left(\mathfrak{L}_{a,b}^{p,q} \left[f_0 + t u, \lambda \right]-\mathfrak{L}_{a,b}^{p,q} \left[f_0 , \lambda \right] \right)\\
\\
\Leftrightarrow  \|~ |x|^a f_0 \|_{L^p(\mathbb{R}^d)}^{1-p} \int_{\mathbb{R}^d} |x|^{ap} |f_0(x)|^{p-2} \mathcal{R} \left(f_0(x) \overline{u(x)} \right) dx + \\
\\
 + \|~ |\omega|^b \widehat{f}_0 \|_{L^q(\mathbb{R}^d)}^{1-q} \int_{\mathbb{R}^d} |\omega|^{bq} |\widehat{f}_0(\omega)|^{q-2} \mathcal{R} \left(\widehat{f}_0(\omega) \overline{\widehat{u}(\omega)} \right) d\omega = \\
 \\
 =\lambda_0 \int_{\mathbb{R}^d} \|V_g f_0 (\cdot, \omega) \|_{L_{\alpha}^r (\mathbb{R}^d)}^{s-r} \left[\int_{\mathbb{R}^d} |V_g f_0 (x, \omega)|^{r-2} \mathcal{R} \left( V_g f_0 (x, \omega) \overline{V_g u (x, \omega)} \right) \langle x \rangle^{\alpha r} dx \right] \langle \omega \rangle^{\beta s} d \omega,
\end{array}
\label{eqBanachCase6}
\end{equation}
where we used the constraint $ \|f_0\|_{M_{\alpha, \beta}^{r,s} (\mathbb{R}^d)}=1$. The result then extends to $L_a^p  (\mathbb{R}^d) \cap \mathcal{F} L_b^q  (\mathbb{R}^d)$ by density.

Since (\ref{eqBanachCase6}) holds for all $u \in L_a^p  (\mathbb{R}^d) \cap \mathcal{F} L_b^q  (\mathbb{R}^d)$, then we get in particular for $u=f_0$:
\begin{equation}
\lambda_0 = \|~| x|^a f_0 \|_{L^p (\mathbb{R}^d)} + \|~| \omega|^b \widehat{f}_0 \|_{L^q (\mathbb{R}^d)} = \mathfrak{L}_{a,b}^{p,q} \left[f_0 \right].
\label{eqBanachCase7}
\end{equation}

If $f_1$ is some other solution of (\ref{eqBanachCase6}) with $\lambda_0 $ replaced by $\lambda_1$ and $f_0$ is the minimizer, we would get:
\begin{equation}
\lambda_0  = \mathfrak{L}_{a,b}^{p,q} \left[f_0 \right] \leq  \mathfrak{L}_{a,b}^{p,q} \left[f_1 \right] = \lambda_1.
\label{eqBanachCase8}
\end{equation}
Finally, if we multiply (\ref{eqBanachCase6}) by $i$ and add it to the same equation with $u$ replaced by $-iu$, we obtain (\ref{eqBanachCase1.1}).
\end{proof}

\subsection{The Hilbert case}

We want to minimize the functional
\[
\mathfrak{F}^{\psi,\phi,m_0} \left[f \right] = \| \psi f \|_{L^2 (\mathbb{R}^d)}^2 + \| \phi \widehat{f} \|_{L^2 (\mathbb{R}^d)}^2= \|f\|_{\psi,\phi,m_0}^2 - \|f \|_{M_{m_0}^2 (\mathbb{R}^d)}^2
\]
in $M_m^2 (\mathbb{R}^d)$, subject to the constraint:
\begin{equation}
\|f \|_{M_{m_0}^2 (\mathbb{R}^d)} = 1.
\label{eqEulerLagrange1}
\end{equation}
We thus optimize the functional
\begin{equation}
\mathfrak{L}^{\psi,\phi,m_0} \left[f , \lambda \right] = \mathfrak{F}^{\psi,\phi,m_0} \left[f  \right] + \lambda \left( 1- \|f \|_{M_{m_0}^2 (\mathbb{R}^d)}^2  \right),
\label{eqEulerLagrange2}
\end{equation}
where $\lambda$ is a Lagrange multiplier. In this section we shall always think of $M_m^2 (\mathbb{R}^d)$ as $\mathcal{B}^{\psi,\phi,m_0}  (\mathbb{R}^d)$, with the inner product $\left(\cdot,\cdot\right)_{\psi,\phi,m_0}$ and the norm $\|\cdot \|_{\psi,\phi,m_0}$.

Before we proceed, let us recall that the operator $A$ defined in eq. (\ref{eqCompactnessHilbert6}) is such that:
\begin{equation}
\left(u, v \right)_{M_{m_0}^2 (\mathbb{R}^d)}= \left(Au,v \right)_{\psi,\phi,m_0},
\label{eqEulerLagrange2.A}
\end{equation}
for all $u,v \in M_m^2 (\mathbb{R}^d)=\mathcal{B}^{\psi,\phi,m_0} (\mathbb{R}^d)$.

\begin{Theorem}\label{TheoremSpectrum}
The operator $A$ is positive-definite, compact and closed. It has empty residual spectrum, $0$ belongs to the continuous spectrum and it is a point of accumulation. Moreover, all remaining spectral values are eigenvalues.
\end{Theorem}

\begin{proof}
From the definition of $A$, we have for all $u,v \in M_m^2 (\mathbb{R}^d)$:
\begin{equation}
(Au, v)_{\psi,\phi,m_0} = (u, v)_{M_{m_0}^2 (\mathbb{R}^d)} = \overline{(v,u)}_{M_{m_0}^2 (\mathbb{R}^d)} = \overline{(Av,u)}_{\psi,\phi,m_0} = (u,Av)_{\psi,\phi,m_0} .
\label{eqSpectrum1}
\end{equation}
Hence $A=A^{\ast}$.

Similarly
\begin{equation}
(Au,u)_{\psi,\phi,m_0} = \|u \|_{M_{m_0}^2 (\mathbb{R}^d)}^2 >0,
\label{eqSpectrum2}
\end{equation}
for all $u \in M_m^2 (\mathbb{R}^d) \backslash \left\{0 \right\}$.
Consequently, $A$ is positive-definite.

That $A$ is closed is a simple consequence of the fact that it is bounded and defined on the whole of $M_m^2 (\mathbb{R}^d)$.

Next, we prove compactness. Let $\left(u_n \right)_{n \in \mathbb{N}}$ be a
bounded sequence in $M_m^2 (\mathbb{R}^d)$:
\begin{equation}
\|u_n\|_{\psi,\phi,m_0} \le C
\label{eqSpectrum3}
\end{equation}
for some constant $C>0$ and all $n \in \mathbb{N}$. Since $M_m^2 (\mathbb{R}^d) \subset
\subset M_{m_0}^2 (\mathbb{R}^d)$, $\left(u_n \right)_{n \in \mathbb{N}}$ has a subsequence $\left(v_n \right)_{n \in \mathbb{N}}$ which converges in $M_{m_0}^2 (\mathbb{R}^d)$.
It follows that
\begin{equation}
\begin{array}{c}
\|Av_n-Av_m\|_{\psi,\phi,m_0}^2 = \left(v_n-v_m,A(v_n-v_m) \right)_{M_{m_0}^2(\mathbb{R}^d)} \leq\\
\\
\leq \| A (v_n-v_m)\|_{M_{m_0}^2(\mathbb{R}^d)} \| v_n-v_m \|_{M_{m_0}^2(\mathbb{R}^d)}  \leq \\
\\
\leq \| A (v_n-v_m)\|_{M_m^2 (\mathbb{R}^d)}
 \|v_n-v_m \|_{M_{m_0}^2(\mathbb{R}^d)} \asymp\\
 \\
\asymp \| A (v_n-v_m)\|_{\psi,\phi,m_0}
 \|v_n-v_m \|_{M_{m_0}^2(\mathbb{R}^d)}\leq\\
 \\
 \leq 2 C \|A \|_{Op} ~ \|v_n-v_m \|_{M_{m_0}^2(\mathbb{R}^d)}~,
\end{array}
\label{eqSpectrum4}
\end{equation}
where we used (\ref{eqSpectrum3}) in the last inequality. This shows that $\left(Av_n \right)_{n \in \mathbb{N}}$ is a Cauchy sequence.
 Since $M_m^2 (\mathbb{R}^d)$ is complete, we conclude that
 $\left(Av_n \right)_{n \in \mathbb{N}}$ converges. Since the bounded sequence
 $\left(u_n \right)_{n \in \mathbb{N}}$ was chosen arbitrarily, the operator $A$
 is compact.

That all non-zero elements of the spectrum are eigenvalues is an immediate consequence of the fact that $A$ is compact. Eq. (\ref{eqSpectrum2}) shows that $A$ is injective. Since $A$ is compact and injective, we conclude that $0$ is in the continuous spectrum and the residual spectrum is empty.

If the spectrum of $A$ were finite, then $A$ would have to be of finite rank. But since $A: M_m^2 (\mathbb{R}^d) \to Ran(A)$ is bijective, this is impossible. Hence, the spectrum is infinite and $0$ must be an accumulation point.
\end{proof}

We are now in a position to obtain the Euler-Lagrange equation for the minimizer $f_0$.

\begin{Theorem}\label{TheoremEulerLagrangeEquation1}
The minimizer $f_0$ of $\mathfrak{F}^{\psi,\phi,m_0} $ in $M_m^2 (\mathbb{R}^d)$, subject to the constraint (\ref{eqEulerLagrange1}) is a solution of the eigenvalue equation:
\begin{equation}
A f_0 = \frac{1   }{\lambda +1} f_0.
\label{eqEulerLagrange4}
\end{equation}
Moreover, we have that
\begin{equation}
\lambda = \mathfrak{F}^{\psi,\phi,m_0} \left[f_0 \right]
\label{eqEulerLagrange5}
\end{equation}
is such that $\frac{1}{\lambda+1}$ is the largest eigenvalue of the operator $A$.
\end{Theorem}

\begin{proof}
Notice that we can reexpress the functional (\ref{eqEulerLagrange2}) as:
\begin{equation}
\mathfrak{L}^{\psi,\phi,m_0} \left[f , \lambda \right] = \left(f, f \right)_{\psi,\phi,m_0} -\left(A f, f \right)_{\psi,\phi,m_0} + \lambda \left( 1- \left(A f, f \right)_{\psi,\phi,m_0}\right),~
\label{eqEulerLagrange6}
\end{equation}

The stationarity condition is:
\begin{equation}
f_0 - A f_0 - \lambda A f_0=0,
\label{eqEulerLagrange7}
\end{equation}
which yields (\ref{eqEulerLagrange4}).

If we compute the inner product $\left(\cdot,\cdot \right)_{\psi,\phi,m_0}$ of the previous equation with $f_0$ and use the constraint condition (\ref{eqEulerLagrange1}), we obtain:
\begin{equation}
\begin{array}{c}
\left(f_0,f_0\right)_{\psi,\phi,m_0} - \left(A f_0,f_0\right)_{\psi,\phi,m_0} - \lambda \left(A f_0,f_0\right)_{\psi,\phi,m_0}=0\\
\\
\Leftrightarrow \|f_0\|_{\psi,\phi,m_0}^2 - \|f_0\|_{M_{m_0}^2 (\mathbb{R}^d)}^2 - \lambda \|f_0\|_{M_{m_0}^2 (\mathbb{R}^d)}^2=0\\
\\
\Leftrightarrow \mathfrak{F}^{\psi,\phi,m_0} \left[f_0\right] - \lambda=0,
\end{array}
\label{eqEulerLagrange8}
\end{equation}
which is (\ref{eqEulerLagrange5}).

It remains to prove that $\frac{1}{\lambda+1}$ is the largest eigenvalue of $A$. Let $\mu \in \mathbb{R}$ be some other eigenvalue of $A$ and $v_{\mu} \in  M_m^2(\mathbb{R}^d)$ an eigenvector associated with it, which we assume to be in $\Omega$:
\begin{equation}
A v_{\mu}= \mu v_{\mu}, \hspace{1 cm} \| v_{\mu} \|_{M_{m_0}^2 (\mathbb{R}^d)}=1.
\label{eqEulerLagrange10}
\end{equation}
From the previous equation, we conclude that:
\begin{equation}
(1-\mu)Av_{\mu} = \mu \left(v_{\mu}-A v_{\mu} \right),
\label{eqEulerLagrange11}
\end{equation}
 and if we compute the inner product of this equation with $v_{\mu}$, we obtain:
 \begin{equation}
\begin{array}{c}
(1-\mu) \left(Av_{\mu},v_{\mu} \right)_{\psi,\phi,m_0}= \mu \left(\left(v_{\mu},v_{\mu} \right)_{\psi,\phi,m_0} -\left(Av_{\mu},v_{\mu} \right)_{\psi,\phi,m_0} \right)\\
\\
\Leftrightarrow (1-\mu) \|v_{\mu} \|_{M_{m_0}^2 (\mathbb{R}^d)}^2= \mu \left( \|v_{\mu} \|_{\psi,\phi,m_0}^2 - \|v_{\mu} \|_{M_{m_0}^2 (\mathbb{R}^d)}^2\right) \\
\\
\Leftrightarrow 1-\mu= \mu \mathfrak{F}^{\psi,\phi,m_0} \left[v_{\mu} \right].
\end{array}
\label{eqEulerLagrange12}
\end{equation}
It follows that:
\begin{equation}
\mathfrak{F}^{\psi,\phi,m_0} \left[v_{\mu} \right] \geq \mathfrak{F}^{\psi,\phi,m_0} \left[f_0 \right] \Leftrightarrow \frac{1}{\mu} -1 \geq \lambda \Leftrightarrow \mu \leq \frac{1}{\lambda +1},
\label{eqEulerLagrange13}
\end{equation}
which concludes the proof.
\end{proof}

It may be more convenient to express the Euler-Lagrange equations in terms of the inverse $A^{-1}$ rather than $A$.

\begin{Theorem}\label{TheoremInverse}
$A^{-1}$ is densely defined in $M_m^2 (\mathbb{R}^d)$, closed and positive-definite. Its spectrum consists only of eigenvalues which can be written as a sequence $0 < \nu_1 \leq \nu_2 \leq \cdots$, with $\nu_j \to + \infty$. Moreover, all the eigenspaces are finite dimensional.
\end{Theorem}

\begin{proof}
We start by proving that $Ran(A)$ is dense in $M_m^2(\mathbb{R}^d)$. Since $m \in \mathcal{P} (\mathbb{R}^{2d})$, then $\mathcal{S}(\mathbb{R}^{d}) \subset M_m^2(\mathbb{R}^d) $. Let $u \in M_m^2 (\mathbb{R}^d)$ be such that:
\begin{equation}
0= \left(Av,u \right)_{\psi,\phi,m_0} = \left(v,u \right)_{M_{m_0}^2 (\mathbb{R}^d)},
\label{eqTheoremInverse1}
\end{equation}
for all $v \in \mathcal{S}(\mathbb{R}^{d})$. Since $\mathcal{S}(\mathbb{R}^{d})$ is dense in $M_{m_0}^2 (\mathbb{R}^d)$, we conclude that $u=0$, and thus $\left\{Av: ~v \in \mathcal{S}(\mathbb{R}^{d}) \right\}$ is dense in $ \mathcal{B}^{\psi,\phi,m_0} (\mathbb{R}^d)=M_m^2(\mathbb{R}^d)$.

Under these circumstances and taking into account Theorem \ref{TheoremSpectrum}, we conclude that $A^{-1}$ is densely defined, closed and that $(A^{-1})^{\ast}=(A^{\ast})^{-1}$. Thus $A^{-1}$ is self-adjoint and positivity follows immediately.

The statements regarding the spectrum are also an immediate consequence of Theorem \ref{TheoremSpectrum}. We just remark that $0$ is a regular value of $A^{-1}$, since its inverse $A$ exists, is bounded and defined on the whole $ M_m^2(\mathbb{R}^d)$.
\end{proof}

\begin{Theorem}\label{TheoremEulerLagrangeInverse}
The minimizer $f_0$ is an eigenvector of $A^{-1}$ associated with the eigenvalue $\lambda +1=\mathfrak{F}^{\psi,\phi,m_0} \left[f_0 \right] +1$, which is the smallest eigenvalue of $A^{-1}$.
\end{Theorem}

\begin{proof}
The proof follows {\it mutatis mutandis} that of Theorem \ref{TheoremEulerLagrangeEquation1}.
\end{proof}

\begin{Example}\label{ExampleSchrodingerEq}
As a particular example, consider a $d=1$ system with $M_{m_0}^2 (\mathbb{R})=L^2 (\mathbb{R})$ ($m_0\equiv 1$), $\psi (x)$ and $\phi(\omega) = \omega$
admissible weights.

We thus have, for all $u \in Ran(A)$, $v \in \mathcal{B}^{\psi, \omega,1} (\mathbb{R})$:
\begin{equation}
\begin{array}{c}
\left(A^{-1}u,v \right)_{L^2 (\mathbb{R})}= \left(u,v \right)_{\psi, \omega,1}\\
\\
\Leftrightarrow \int_{\mathbb{R}}\left(A^{-1} u\right) (x) \overline{v(x)} dx = \int_{\mathbb{R}}\left(1+ | \psi(x)|^2  \right) u(x)\overline{v(x)}  dx +  \int_{\mathbb{R}}| \omega|^2 \widehat{u} (\omega) \overline{\widehat{v} (\omega)} d \omega \\
\\
\Leftrightarrow \int_{\mathbb{R}}\left(A^{-1} u\right) (x) \overline{v(x)} dx = \int_{\mathbb{R}}\left(1+ | \psi(x)|^2  \right) u(x)\overline{v(x)}  dx +  \int_{\mathbb{R}} \mathcal{D}u (x) \overline{\mathcal{D}v (x)} d x,
\end{array}
\label{eqExample2}
\end{equation}
where $\mathcal{D}$ denotes the distributional derivative.

 Since $\mathcal{S} (\mathbb{R} ) \subset \mathcal{B}^{\psi, \phi,1}
(\mathbb{R})$, we conclude from Theorem \ref{TheoremEulerLagrangeInverse} and (\ref{eqExample2}) that the minimizer $f_0$ satisfies:
\begin{equation}
- \frac{1}{4 \pi^2} \mathcal{D}^2 f_0 (x) + | \psi (x)|^2 f_0 (x) = \lambda f_0 (x),
\label{eqExample3}
\end{equation}
with $\lambda$ the smallest eigenvalue of $A^{-1} - I=- \frac{1}{4 \pi^2} \mathcal{D}^2+ | \psi (x)|^2$.

 If, in addition, $\psi (x)$ is such that $| \psi (x) |^2 \in C^2 (\mathbb{R})$, then a well known theorem shows that the distributional and classical solutions coincide, so that $f_0$ is a solution of the time-independent Schr\"odinger equation
 \begin{equation}
 - \frac{1}{4 \pi^2}  f_0^{\prime \prime} (x) + | \psi (x)|^2 f_0 (x) = \lambda f_0 (x).
 \label{eqExample4}
 \end{equation}
If $\psi (x) =x$, then all the aforementioned conditions are satisfied, and we obtain the Schr\"odinger equation for the simple harmonic oscillator:
\begin{equation}
- \frac{1}{4 \pi^2}  f_0^{\prime \prime} (x) + x^2 f_0 (x) = \lambda  f_0 (x).
\label{eqExample5}
\end{equation}
It is well known that a solution associated with the lowest eigenvalue is:
\begin{equation}
f_0 (x) = 2^{1/4 } e^{- \pi x^2},
\label{eqExample6}
\end{equation}
which corresponds to
\begin{equation}
\lambda= \frac{1}{2 \pi}.
\label{eqExample7}
\end{equation}
And we thus obtain the inequality
\begin{equation}
 \|x u \|_{L^2(\mathbb{R})}^2 + \| \omega \widehat{u} \|_{L^2(\mathbb{R})}^2 \geq \frac{1}{2 \pi}
\label{eqExample8}
\end{equation}
for all $u$ with $\|u\|_{L^2(\mathbb{R})}=1$, for which the left-hand side makes sense. This is equivalent to the standard Heisenberg uncertainty principle.

Notice that we have chosen to obtain the minimizer with the inverse operator $A^{-1}$ instead of the compact operator $A$. The latter alternative would have been much harder to solve \cite{Bongioanni,Cappiello}.

\end{Example}

\section*{Acknowledgements}

The work of N.C. Dias and J.N. Prata is supported by the Portuguese Science Foundation (FCT) grant PTDC/MAT-CAL/4334/2014.

\end{document}